\input ./style/arxiv-general.cfg
\documentclass[aap,MSNbibl,dvips]{arximspdf}
\makeatletter
   \@ifpackageloaded{graphicx}{}{\usepackage{graphicx}}
\makeatother
\usepackage{mathbh}

\doi{10.1214/14-AAP1047} 
\volume{25}
\issue{4}
\pubyear{2015}
\firstpage{2215}
\lastpage{2262}
\docsubty{FLA}

\makeatletter
\newcommand{\GW}{\mathrm{GW}}
\newcommand{\st}{s,t}
\newcommand{\br}{\mathrm{br}}

\newcommand{\lleft}{\left}
\newcommand{\rrvert}{\vert}
\newcommand{\rright}{\right}
\newcommand{\llvert}{\vert}
\newtheorem{teo}{Theorem}[section]
\newtheorem{prop}[teo]{Proposition}
\newtheorem{lem}[teo]{Lemma}
\newtheorem{cor}[teo]{Corollary}
\newproclaim{defn}[teo]{Definition}
\newcommand{\halpha}{\hat{\alpha}}
\makeatother

\begin{document}
\begin{frontmatter}

\title{The vertex-cut-tree of Galton--Watson trees converging to a
stable tree}
\runtitle{Vertex-cut-tree}

\begin{aug}
\author[A]{\fnms{Daphn\'{e}}~\snm{Dieuleveut}\corref{}\ead[label=e1]{daphne.dieuleveut@math.u-psud.fr}}
\runauthor{D. Dieuleveut}
\affiliation{Universit\'{e} Paris-Sud}
\address[A]{Equipe de Probabilit\'{e}s,\\
\quad Statistiques et Mod\'{e}lisation\\
Universit\'{e} Paris-Sud\\
B\^{a}timent 430\\
91405 Orsay Cedex\\
France\\
\printead{e1}} 
\end{aug}

\received{\smonth{12} \syear{2013}}
\revised{\smonth{5} \syear{2014}}

%
\begin{abstract}
We consider a fragmentation of discrete trees where the internal
vertices are deleted
independently at a rate proportional to their degree. Informally, the
associated cut-tree represents the genealogy of the nested connected
components created by this process. We essentially work in the setting
of Galton--Watson trees with offspring distribution belonging to the
domain of attraction of a stable law of index $\alpha\in(1,2)$. Our
main result is that, for a sequence of such trees $\mathcal{T}_n$ conditioned
to have size $n$, the corresponding rescaled cut-trees converge in
distribution to the stable tree of index $\alpha$, in the sense
induced by the Gromov--Prokhorov topology. This gives an analogue of a
result obtained by Bertoin and Miermont in the case of Galton--Watson
trees with finite variance.
\end{abstract}

%
\end{frontmatter}

\section{Introduction and main result}\label{sec1}

Fragmentations of random trees were first introduced in the work of
case of Meir and Moon \cite{MM} as a recursive random edge-deletion
process on discrete trees. Since then, it has been recognized that
fragmentations of discrete and continuous trees appear in several
natural contexts; see, for example, \cite{BerFires,DroSch} for a
connection with forest fire models, \cite{AldPit,BerFragMB} for
fragmentations of the Brownian tree \cite{AldCRT3} and its relation to
the additive coalescent, and \cite{AbrDV,Mi03,Mi05} for fragmentations
of the stable tree of index $\alpha\in(1,2)$ \cite{DuqLG02}. The
fragmentations considered in the two last cases, which arise naturally
in the setting of Brownian and stable trees, are self-similar
fragmentations as studied by Bertoin \cite{BerFragAS}, whose
characteristics are explicitly known.

Several recent articles investigated the question of the asymptotic
distribution of the number of cuts needed to isolate a specific vertex,
for various classes of random trees. In specific cases, Panholzer \cite
{Pan06} showed that the Rayleigh distribution arises naturally as a
limit in this context, and Janson \cite{Jan} showed that this limiting
result holds for general Galton--Watson trees with a finite variance
offspring distribution, using a method of moments. He also established
a connection to the Brownian tree, which is natural since the Rayleigh
distribution is the law of the distance between two uniformly chosen
vertices in the CRT. Later, Addario-Berry, Broutin and Holmgren \cite
{ABBH} provided a different proof giving a more concrete connection to
the Brownian tree. Bertoin and Miermont \cite{BerMi} then studied the
genealogy of the cutting procedure in itself, which is related to the
problem of the isolation of several vertices rather than just the root
(certain of these ideas were implicitly present in former papers,
including \cite{ABBH,BerFires}). This allows to code the discrete
cutting procedure in terms of a ``cut-tree,'' whose scaling limit is
shown to be a Brownian tree that describes in some sense the genealogy
of the Aldous--Pitman fragmentation~\cite{AldPit}.

Note that the results of \cite{ABBH}, by introducing a reversible
transformation of the Brownian tree, can be understood as building the
``first branch'' of the limiting cut-tree, the latter being a kind of
iteration \textit{ad libitum} of this transformation. This
transformation was extended in \cite{AbrDFARPLT} in the context of a
fragmentation of stable trees. The main goal of the present work is to
show that the approach of Bertoin and Miermont \cite{BerMi} can also
be adapted to Galton--Watson trees with offspring distribution in the
domain of attraction of a non-Gaussian stable law, showing the
convergence of the whole discrete cut-tree to a limiting stable tree.
This gives in passing a natural definition of the continuum cut-tree for
the fragmentation studied in \cite{Mi05}.

Let us describe more precisely the result of \cite{BerMi} we are
interested in. Consider a sequence of Galton--Watson trees $\mathcal{T}_n$,
conditioned to have exactly $n$ edges, with critical offspring
distribution having finite variance $\sigma^2$. The associated
cut-trees $\operatorname{Cut}(\mathcal{T}_n)$ describe the genealogy
of the fragments
obtained by deleting the edges in a uniform random order. It is well
known that the rescaled trees $(\sigma/\sqrt{n}) \cdot\mathcal{T}_n$
converge in distribution to the Brownian tree $\mathcal{T}$; see \cite
{AldCRT3} for the convergence of the associated contour functions,
which implies that this convergence holds for the commonly used
Gromov--Hausdorff topology, and for the Gromov--Prokhorov topology. In the
present work, we will mainly use the latter. Bertoin and Miermont
showed that there is in fact the joint convergence
\[
\biggl(\frac{\sigma}{\sqrt{n}} \mathcal{T}_n, \frac{1}{\sigma
\sqrt{n}}
\operatorname{Cut}(\mathcal{T}_n) \biggr) \mathop{ \longrightarrow}_{n
\rightarrow\infty}^{(d)}\, \bigl(\mathcal{T}, \operatorname{Cut}
(\mathcal{T}) \bigr),
\]
where $\operatorname{Cut}(\mathcal{T})$ is the so-called cut-tree of
$\mathcal{T}$. Informally,
$\operatorname{Cut}(\mathcal{T})$ describes the genealogy of the
fragments obtained by
cutting $\mathcal{T}$ at points chosen according to a~Poisson point
process on
its skeleton. Moreover, $\operatorname{Cut}(\mathcal{T})$ has the
same law as~$\mathcal{T}$. 

Our goal is to show an analogue result in the case where the $\mathcal{T}_n$
are Galton--Watson trees with offspring distribution belonging to the
domain of attraction of a stable law of index $\alpha\in(1,2)$, and
$\mathcal{T}$ is the stable tree of index $\alpha$. For the stable
tree, a
self-similar fragmentation arises naturally by splitting at branching
points with a rate proportional to their ``width,'' as shown in \cite
{Mi05}. This will lead us to modify the edge-deletion mechanism for the
discrete trees, so that the rate at which internal vertices are removed
increases with their degree. Therefore, we call \emph
{edge-fragmentation} the fragmentation studied in \cite{BerMi}, and
\emph{vertex-fragmentation} our model. Note that more general
fragmentations of the stable tree can be constructed by splitting both
at branching points and at uniform points of the skeleton, as in \cite
{AbrDV}. However, these fragmentations are not self-similar (see \cite
{Mi05}), and will not be studied here.

In the rest of the \hyperref[sec1]{Introduction}, we will describe our setting more
precisely and give the exact definition of the cut-trees, both in the
discrete and the continuous cases. This will enable us to state our
main results in Section~\ref{SThm}.

\subsection{Vertex-fragmentation of a discrete tree} \label{SFragTn}

We begin with some notation. Let~$\mathbb{T}$ be the set of all finite
plane rooted trees. For every $T \in\mathbb{T}$, we call $E (T)$ the
set of edges of $T$, $V (T)$ the set of vertices of $T$, and $\rho(T)$
the root-vertex of $T$. For each vertex $v \in V (T)$, $\deg(v,T)$
denotes the number of children of $v$ in $T$ (or $\deg v$, if this
notation is not ambiguous), and for each edge $e \in E (T)$, $e^-$
(resp.,~$e^+$) denotes the extremity of $e$ which is closest to (resp.,
furthest away from) the root.

For any tree $T$ with $n$ edges, we label the vertices of $T$ by $v_0,
v_1, \ldots, v_n$, and the edges of $T$ by $e_1, \ldots, e_n$, in the
depth-first order. Note that the planar structure of $T$ gives an order
on the offspring of each vertex, say ``from left to right,'' hence the
depth-first order is well defined. With this notation, we have $v_j =
e_j^+$ for all $j \in\{1,\ldots,n\}$.

We let $T \in\mathbb{T}$ be a finite tree with $n$ edges. We consider
a discrete-time fragmentation on $T$, which can be described as follows:
\begin{itemize}
\item at each step, we mark a vertex of $T$ at random, in such a way
that the probability of marking a given vertex $v$ is proportional to
$\deg v$;
\item when a vertex $v$ is marked, we delete all the edges $e$ such
that $e^- = v$.
\end{itemize}
Note that the total number of steps $N$ is at most $n$. To keep track
of the genealogy induced by this edge-deletion process, we introduce a
new structure called the cut-tree of $T$, denoted by $\operatorname{Cut}_{\mathrm{v}}(T)$.

For all $r \in\{1,\ldots,N\}$, we let $v (r)$ be the vertex which
receives a mark at step~$r$, $E_r = \{ e \in E(T)\dvtx  e^- = v (r)
\}$ be the set of the edges which are deleted at step~$r$, $k_r =
\llvert E_r\rrvert$, and $D_r = \{ i \in\{1,\ldots,n\}\dvtx  e_i \in\bigcup
_{r' \leq r} E_{r'} \}$. We say that $j \sim_r j'$ if and only
if $e_j$ and $e_{j'}$ are still connected in the forest obtained from
$T$ by deleting the edges in $D_r$. Thus, $\sim_r$ is an equivalence
relation on $\{1,\ldots,n\} \setminus D_r$. The family of the
equivalence classes (without repetition) of the relations $\sim_r$ for
$r = 1,\ldots,N$ forms the set of internal nodes of $\operatorname{Cut}_{\mathrm{v}}(T)$. The
initial block $\{1,\ldots,n\}$ is seen as the root, and the leaves of
$\operatorname{Cut}_{\mathrm{v}}(T)$ are given by $1, \ldots, n$. We stress that we distinguish
the leaves $i$ and the internal nodes $\{i\}$.

We now build the cut-tree $\operatorname{Cut}_{\mathrm{v}}(T)$ inductively. At the $r$th step,
we let $B$ be the equivalence class for $\sim_{r-1}$ containing the
indices $i$ such that $e_i \in E_r$. Deleting the edges in $E_r$ splits
the block $B$ into $k'_r$ equivalence classes $B_1, \ldots, B_{k'_r}$
for $\sim_r$, with $k'_r \leq k_r +1$. We draw $k'_r$ edges between
$B$ and the sets $B_1, \ldots, B_{k'_r}$, and $k_r$ edges between $B$
and the leaves $i$ such that $e_i \in E_r$. Thus, the graph-distance
between the leaf $i$ and the root in $\operatorname{Cut}_{\mathrm{v}}(T)$ is the number of cuts
in the component of $T$ containing the edge $e_i$ before $e_i$ itself
is removed. Note that $\operatorname{Cut}_{\mathrm{v}}(T)$ does not have a natural planar
structure, but that the actual embedding does not intervene in our
work. Figure~\ref{FCutT} gives an example of this construction for a
tree $T$ with 16 edges.

If $T$ is a random tree, the fragmentation of $T$ and the cut-tree
$\operatorname{Cut}_{\mathrm{v}}(T)$ are defined similarly, by conditioning on $T$ and
performing the above construction.

Note that, equivalently, we could mark the edges of $T$ in a uniform
random order, and delete all the edges $e$ such that $e^- = e_i^-$, as
soon as $e_i$ is marked. The cut-tree $\operatorname{Cut}_{\mathrm{v}}(T)$ would then be
obtained by performing the same construction with $E_r = \{ e \in
E(T)\dvtx  e^- = e_{i_r}^-\}$. This procedure sometimes adds ``neutral
steps,'' which have no effect on the fragmentation, but this does not
change the cut-tree. It will sometimes be more convenient to work with
this point of view, for example, in Sections~\ref{SModDist} and~\ref{SBrownianCase}.

%
\begin{figure}

\includegraphics{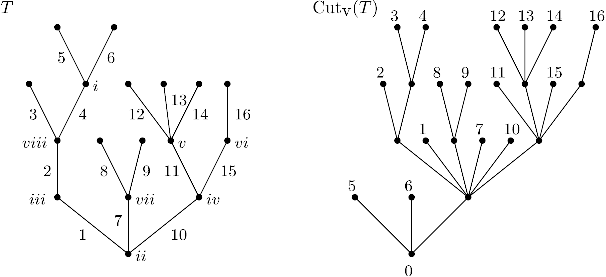}

\caption{The cut-tree $\operatorname{Cut}_{\mathrm{v}}(T)$ of a tree $T$. The order of deletion
of the internal vertices of $T$ is indicated in Roman numerals. The
correspondence between the edges of $T$ and the leaves of $\operatorname{Cut}_{\mathrm{v}}(T)$
is indicated in Arabic numerals.}\label{FCutT}
\end{figure}

\subsection{Fragmentation and cut-tree of the stable tree of index \texorpdfstring{$\alpha\in(1,2)$}{$alpha in(1,2)$}} \label{SFragT}

Following Duquesne and Le Gall (see, e.g., \cite{DuqLG05}), we see
stable trees as random rooted $\mathbb{R}$-trees.

%
\begin{defn}
A metric space $(T,d)$ is an $\mathbb{R}$-tree if, for every $u,v \in T$:
\begin{itemize}
\item There exists a unique isometric map $f_{u,v}$ from $[0,d(u,v)]$
into $T$ such that $f_{u,v}(0) = u$ and $f_{u,v}(d(u,v)) = v$.
\item For any continuous injective map $f$ from $[0,1]$ into $T$, such
that $f(0)=u$ and $f(1)=v$, we have
\[
f\bigl([0,1]\bigr) = f_{u,v} \bigl(\bigl[0,d(u,v)\bigr]\bigr):=
[\![ u, v ]\!].
\]
\end{itemize}
A rooted $\mathbb{R}$-tree is an $\mathbb{R}$-tree $(T,d,\rho)$ with
a distinguished
point $\rho$ called the root.
\end{defn}

The trees we will work with can be seen as $\mathbb{R}$-trees coded by
continuous functions from $[0,1]$ into $\mathbb{R}_+$, as in \cite{DuqLG05}.
In particular, the stable tree $(\mathcal{T},d)$ of index $\alpha$ is
the $\mathbb{R}
$-tree coded by the excursion of length $1$ of the height process
$H^{(\alpha)}$, defined as follows in \cite{DuqLG02}. Let $X^{(\alpha
)}$ be a stable spectrally positive L\'{e}vy process with parameter
$\alpha$, whose normalization will be prescribed in Section~\ref{SLimLocale}. For every $t >0$, let $\widehat{X}{}^{(\alpha,t)}$ be the
process defined by
\begin{eqnarray*}
\widehat{X}{}^{(\alpha,t)}_s = \cases{ X^{(\alpha)}_t
- X^{(\alpha)}_{(t-s)^-}, &\quad if $0 \leq s < t$,
\vspace*{2pt}\cr
X^{(\alpha)}_t, &\quad if $s=t$,}
\end{eqnarray*}
and write $\hat{S}^{(\alpha,t)}_s = \sup_{0 \leq r \leq s} \hat
{X}^{(\alpha,t)}_r$ for all $r \in[0,t]$.

%
\begin{defn}
The height process $H^{(\alpha)}$ is the real-valued process such that
$H_0 = 0$ and, for every $t>0$, $H_t$ is the local time at level $0$ at
time $t$ of the process~$\widehat{X}{}^{(\alpha,t)}-\hat{S}^{(\alpha,t)}$.
\end{defn}

The normalization of local time, and the proof of the existence of a
continuous modification of this process, are given in \cite{DuqLG02}, Section~1.2. This definition of $\mathcal{T}$ allows us to introduce
the canonical projection $p\dvtx [0,1] \rightarrow\mathcal{T}$. We endow
$\mathcal{T}$ with
a probability mass-measure $\mu$ defined as the image of the Lebesgue
measure on $[0,1]$ under $p$, and say that the rot of $\mathcal{T}$ is the
unique point which has height $0$.

For the fragmentation of the stable tree, we will use a process
introduced and studied by Miermont in \cite{Mi05}, which consists in
deleting the nodes of $\mathcal{T}$ in such a way that the
fragmentation is
self-similar. We first recall that the multiplicity of a point $v$ in
an $\mathbb{R}$-tree $T$ can be defined as the number of connected components
of $T \setminus\{v\}$. To be consistent with the definitions of
Section~\ref{SFragTn}, we define the degree of a point as its
multiplicity minus $1$, and say that a branching point of $T$ is a
point $v$ such that $\deg(v,T) \geq2$. Duquesne and Le Gall have shown
in \cite{DuqLG05}, Theorem 4.6, that $\mbox{a.s.}$ the branching
points in $\mathcal{T}
$ form a countable set, and that these branching points have infinite
degree. We let $\mathcal{B}$ denote the set of these branching points.
For any $b \in\mathcal{B}$, one can define the local time, or width
of $b$ as the almost sure limit
\[
L (b) = \lim_{\varepsilon\rightarrow0^+} \varepsilon^{-1} \mu\bigl\{ v
\in\mathcal{T}\dvtx  b \in[\![ \rho, v ]\!], d (b,v) < \varepsilon
\bigr\},
\]
where $\rho$ is the root of the stable tree $\mathcal{T}$. The
existence of
this quantity is justified in \cite{Mi05}, Proposition 2, (see also
\cite{DuqLG05}).

We can now describe the fragmentation we are interested in.
Conditionally on~$\mathcal{T}$, we let $(t_i,b_i)_{i \in I}$ be the family
(indexed by a countable set $I$) of the atoms of a Poisson point
process with intensity $dt \otimes\sum_{b \in\mathcal{B}} L (b)
\delta_b (dv)$ on $\mathbb{R}_+ \times\mathcal{B}$.\vspace*{2pt} Seeing these
atoms as
marks on the branching points of $\mathcal{T}$, we let $\overline
{\mathcal{T}} (t) = \mathcal{T}
\setminus\{b_i\dvtx  t_i \leq t\}$.

For every $x \in\mathcal{T}$, we let $\mathcal{T}_x (t)$ be the
connected component of
$\overline{\mathcal{T}} (t)$ containing~$x$, with the convention that
$\mathcal{T}_x (t) =
\varnothing$ if $x \notin\overline{\mathcal{T}} (t)$. We\vspace*{1pt} also let
$\mu_x
(t) = \mu
(\mathcal{T}_x (t))$. Adding a distinguished point $0$ to $\mathcal
{T}$, we define a
function $\delta$ from $(\mathcal{T}\sqcup\{0\})^2$ into $\mathbb
{R}_+ \cup\{\infty
\}$, such that for all $x,y \in\mathcal{T}$,
\begin{eqnarray*}
\delta(0,0) & =& 0, \qquad\delta(0,x) = \delta(x,0) = \int_0^{\infty}
\mu_x (t) \,dt,
\\
\delta(x,y) &=& \int_{t (x,y)}^{\infty} \bigl(
\mu_x (t) + \mu_y (t) \bigr) \,dt,
\end{eqnarray*}
where $t (x,y):= \inf\{t \in\mathbb{R}_+\dvtx  \mathcal{T}_x (t) \neq
\mathcal{T}_y (t)\}$ is
$\mbox{a.s.}$ finite. We think of $\delta$ as our new ``distance'' in the
cut-tree. This definition might seem surprising, but the results of
Section~\ref{SModDist} will show that it provides an analogue of the
distance we defined in the discrete case, in terms of number of cuts;
as will be explained in Section~\ref{SEqldeltad}, it also has a
natural interpretation as a time-change between two fragmentation
processes of the stable tree, studied in \cite{Mi03} and \cite{Mi05}.
The role of the extra point $0$ in our (time-changed) fragmentation
will be similar to the role played by the root of $\mathcal{T}$ in the
``fragmentation at heights'' which will be introduced in Section~\ref{SEqldeltad}.

A first idea would be to build the vertex-cut-tree $\operatorname{Cut}_{\mathrm{v}}(\mathcal
{T})$ as a
completion of $(\mathcal{T}\sqcup\{0\}, \delta)$. However, making
this idea
rigorous is difficult, since it is not clear whether $\delta$ is
$\mbox{a.s.}
$ finite, and defines a distance on $\mathcal{T}\sqcup\{0\}$. We will instead
use an approach introduced by Aldous, which consists in building a
continuous random tree such that the subtrees determined by $k$
randomly chosen leaves have the right distribution. To this end, we use
the conditions given by Aldous in \cite{AldCRT3}, Theorem~3.

Set $\xi(0) = 0$, and let $(\xi(i))_{i \in\mathbb{N}}$ be an i.i.d.
sequence distributed according to $\mu$, conditionally on $\mathcal
{T}$. The
key argument of our construction is the identity in law
\[
\bigl( \delta\bigl(\xi(i), \xi(j)\bigr) \bigr)_{i,j \geq0} \stackrel{(d)}
{=} \bigl( d \bigl(\xi(i+1), \xi(j+1)\bigr) \bigr)_{i,j \geq0},
\]
which will be proven in Section~\ref{SEqldeltad}. In particular, it
implies that almost surely, for all $i,j \geq0$, $\delta(\xi(i), \xi
(j))$ is finite, and that $\delta$ is $\mbox{a.s.}$ a distance on $\{
\xi(i),
i \geq0\}$. This allows us to see the spaces $\mathcal{R} (k):= (\{
\xi(i), 0 \leq i \leq k\}, \delta)$, for all $k \in\mathbb{N}$, as random
rooted trees with $k$ leaves. Using the terminology of Aldous,
$(\mathcal{R} (k), k \in\mathbb{N})$ forms a \emph{consistent}
family of
random rooted trees which satisfies the \emph{leaf-tight condition}:
\[
\min_{1 \leq j \leq k} \delta\bigl(\xi(0), \xi(j)\bigr) \mathop{ \longrightarrow}_{k \rightarrow\infty}^{\mathbb{P}} 0.
\]
Indeed, the second part of Theorem~3 of \cite{AldCRT3} shows that
these conditions hold for the reduced trees $( \{\xi(i), 1 \leq i \leq
k+1\}, d )$. As a consequence, the family $(\mathcal{R} (k), k \in
\mathbb{N}
)$ can be represented as a continuous random tree $\operatorname{Cut}_{\mathrm{v}}(\mathcal
{T})$, and $(
\delta(\xi(i), \xi(j)) )_{i,j \geq0}$ is the matrix of mutual
distances between the points of an i.i.d. sample of
$\operatorname{Cut}_{\mathrm{v}}(\mathcal{T})$.
This tree $\operatorname{Cut}_{\mathrm{v}}(\mathcal{T})$ is called the cut-tree of $\mathcal
{T}$. Note that
$\operatorname{Cut}_{\mathrm{v}}(\mathcal{T})$ depends on $\mathcal{T}$ and on the extra
randomness of the
Poisson process.

\subsection{Fragmentation and cut-tree of the Brownian tree} \label{SFragBr}

We will also work on the Brownian tree $(\mathcal{T}^{\br}, d^{\br},
\rho^{\br})$,
which was defined by Aldous (see \cite{AldCRT3}) as the \mbox{$\mathbb{R}$-}tree
coded by $(H_t)_{0 \leq t \leq1} = (2 B_t)_{0 \leq t \leq1}$, where
$B$ denotes the standard Brownian excursion of length $1$. This tree
can be seen as the stable tree of index $\alpha= 2$ (up to a scale
factor, with the normalization we will use). In particular, we have a
probability mass-measure $\mu^{\br}$ on $\mathcal{T}^{\br}$, defined as
the image
of the Lebesgue measure on $[0,1]$ under the canonical projection. We
also define a length-measure $l$ on $\mathcal{T}^{\br}$, which is the
sigma-finite
measure such that, for all $u,v \in\mathcal{T}^{\br}$, $l ( [\![ u, v ]\!]) = d^{\br} (u,v)$.

The fragmentation of the Brownian tree we consider is the same as in
\cite{BerMi}: conditionally on $\mathcal{T}^{\br}$, we let
$(t_i,b_i)_{i \in I}$
be the family of the atoms of a Poisson point process with intensity
$dt \otimes l(dv)$ on $\mathbb{R}_+ \times\mathcal{T}^{\br}$. As for
the stable tree, we
let $\mathcal{T}^{\br}_x (t)$ be the connected component of $\mathcal
{T}^{\br}\setminus\{b_i\dvtx
t_i \leq t\}$, and $\mu^{\br}_x (t) = \mu^{\br} (\mathcal{T}^{\br}_x
(t))$, for
every $x \in\mathcal{T}^{\br}$. Adding a distinguished point $0$ to
$\mathcal{T}^{\br}$, we
define a function $\delta^{\br}$ on $(\mathcal{T}^{\br}\sqcup\{0\})^2$
such that
for all $x,y \in\mathcal{T}^{\br}$,
\begin{eqnarray*}
\delta^{\br} (0,0) &=& 0, \qquad\delta^{\br} (0,x) =
\delta^{\br} (x,0) = \int_0^{\infty} \mu
^{\br}_x (t) \,dt,
\\
\delta^{\br} (x,y) &=& \int_{t^{\br} (x,y)}^{\infty}
\bigl( \mu^{\br}_x (t) + \mu^{\br}_y
(t) \bigr) \,dt,
\end{eqnarray*}
where $t^{\br} (x,y):= \inf\{t \in\mathbb{R}_+\dvtx  \mathcal{T}^{\br}_x
(t) \neq\mathcal{T}^{\br}_y
(t)\}$ is $\mbox{a.s.}$ finite. As shown in \cite{BerMi}, we can
define a new
tree $\operatorname{Cut}(\mathcal{T}^{\br})$ for which the matrix of
mutual distances between
the points of an i.i.d. sample of $\operatorname
{Cut}(\mathcal{T}^{\br})$ is $(\delta(\xi(i),
\xi(j)) )_{i,j \geq0}$,\vspace*{1pt} where $\xi(0) = 0$ and $(\xi(i))_{i \in
\mathbb{N}
}$ is an i.i.d. sequence distributed according to $\mu^{\br}$,
conditionally on~$\mathcal{T}^{\br}$. Moreover, $\operatorname
{Cut}(\mathcal{T}^{\br})$ has the same law as~$\mathcal{T}^{\br}$.

\subsection{Main results} \label{SThm}

As stated in the \hyperref[sec1]{Introduction}, we mainly work in the setting of
Galton--Watson trees with critical offspring distribution $\nu$, where
$\nu$ is a probability distribution belonging to the domain of
attraction of a stable law of index $\alpha\in(1,2) $. We shall also
assume that $\nu$ is aperiodic. Finally, for a technical reason, we
will need the additional hypothesis
%
%
\begin{equation}
\label{HMajPZ=r} \sup_{r \geq1} \biggl(\frac{r \mathbb{P}(\hat{Z} =
r)}{\mathbb
{P}(\hat{Z} > r)} \biggr) <
\infty,
\end{equation}
where $\hat{Z}$ is a random variable such that $\mathbb{P}(\hat
{Z}=r) = r
\nu(\{r\})$. For example, this is the case if $\nu(\{r\})$ is
equivalent to $c/r^{\alpha+ 1}$ as $n \rightarrow\infty$, for a
constant $c \in(0,\infty)$. In all our work, we shall implicitly work
for values of $n$ such that, for a Galton--Watson tree $T$ with
offspring distribution $\nu$, $\mathbb{P} (\llvert
E(T)\rrvert=n ) \neq0$. We let $\mathcal{T}
_n$ be a $\nu$-Galton--Watson tree, conditioned to have exactly $n$
edges. We let $\delta_n$ denote the graph-distance on $\{0,1,\ldots,n\}
$ induced by $\operatorname{Cut}_{\mathrm{v}}(\mathcal{T}_n)$. We will use the notation
$\rho_n$
for the root of $\mathcal{T}_n$, and $\mu_n$ for the uniform
distribution on
$E (\mathcal{T}_n)$ (by slight abuse, $\mu_n$ will also sometimes be
used for
the uniform distribution on $\{1, \ldots, n\}$).

Our main goal is to study the asymptotic behavior of $\operatorname{Cut}_{\mathrm{v}}(\mathcal
{T}_n)$ as
$n \rightarrow\infty$. To this end, it will be convenient to see
trees as pointed metric measure spaces, and work with the Gromov--Prokhorov
topology on the set of (equivalence classes of) such spaces. Let us
recall a few definitions and facts on these objects (see, e.g., \cite
{GPW08} for details).

A pointed metric measure space is a quadruple $(X, D, m, x)$, where $m$
is a Borel probability measure on the metric space $(X,D)$, and $x$ is
a point of $X$. These objects are considered up to a natural notion of
isometry-equivalence. One says that a sequence $(X_n, D_n, m_n, x_n)$
of pointed measure metric spaces converges in the Gromov--Prokhorov
sense to $(X_{\infty}, D_{\infty}, m_{\infty}, x_{\infty})$ if and
only if the following holds: for $n \in\mathbb{N}\cup\{\infty\}$,
set $\xi
_n(0) = x_n$ and let $\xi_n(1), \xi_n(2), \ldots$ be a sequence of
i.i.d. random variables with law $m_n$, then the vector\break
$(D_n(\xi
_n(i), \xi_n(j))\dvtx  0 \leq i, j \leq k)$ converges in distribution to
$(D_{\infty} (\xi_{\infty} (i);\break \xi_{\infty} (j))\dvtx  0 \leq i, j
\leq k)$ for every $k \geq1$. The space $\mathbb{M}$ of
(isometry-equivalence classes of) pointed measure metric spaces,
endowed with the Gromov--Prokho\-rov topology, is a Polish space.

In this setting, the stable tree $\mathcal{T}$ with index $\alpha$
can be seen
as a scaling limit of the Galton--Watson trees $\mathcal{T}_n, n \in
\mathbb{N}$. More
precisely, we endow the discrete trees $\mathcal{T}_n$ with the associated
graph-distance $d_n$ and the uniform distribution $m_n$ on $V(\mathcal{T}_n)
\setminus\{\rho_n\}$. Note that $m_n$ is uniform on $\{v_1 (\mathcal{T}
_n),\ldots,v_n(\mathcal{T}_n)\}$; by slight abuse, it will sometimes be
identified with the uniform distribution on $\{1,\ldots,n\}$. For any
pointed metric measure space $\mathbf{X} = (X,D,m,x)$ and any $a \in
(0,\infty)$, we let $a \mathbf{X} = (X,a D,m,x)$. With this
formalism, there exists a sequence $(a_n)_{n \in\mathbb{N}}$ such that
%
%
\begin{equation}
\label{ECvTn} \frac{a_n}{n} \mathcal{T}_n \mathop{ \longrightarrow} ^{(d)} \mathcal{T},
\end{equation}
in the sense of the Gromov--Prokhorov topology, and $a_n = n^{1/\alpha}
f(n)$ for a slowly-varying function $f$. This is a consequence of the
convergence of the contour functions associated with the trees
$\mathcal{T}_n$,
shown in \cite{Duq}, Theorem 3.1. We will give a slightly more precise
version of this result in Section~\ref{SCodingTrees}.

We can now state our main result.

%
\begin{teo}\label{TMainThm}
Let $(a_n)_{n \in\mathbb{N}}$ be a sequence such that (\ref{ECvTn}) holds.
Then we have the following joint convergence in distribution:
\[
\biggl(\frac{a_n}{n} \mathcal{T}_n, \frac{a_n}{n}
\operatorname{Cut}_{\mathrm{v}}(\mathcal{T}_n) \biggr) \mathop{ \longrightarrow}_{n \rightarrow\infty}\, \bigl(\mathcal{T}, \operatorname{Cut}_{\mathrm{v}}(\mathcal{T})
\bigr),
\]
where $\mathbb{M}$ is endowed with the Gromov--Prokhorov topology and
$\mathbb{M} \times\mathbb{M}$ has the associated product topology.
Furthermore, the cut-tree $\operatorname{Cut}_{\mathrm{v}}(\mathcal{T})$ has the same
distribution as
$\mathcal{T}$.
\end{teo}

Note that this generalizes Proposition 1.4 of \cite{AbrDBC&GWT}, which
gave the scaling limit of the number of cuts needed to isolate the root
in a stable Galton--Watson tree.

In the following sections, we fix the sequence $(a_n)$. For some of the
preliminary results, we will use a particular choice of this sequence,
detailed in Section~\ref{SLimLocale}. Nevertheless, it is easy to
check that the theorem holds for any equivalent sequence.

To complete this result, we will study the limit of the cut-tree
obtained for the vertex-fragmentation, in the case where the offspring
distribution $\nu$ has finite variance (still assuming that $\nu$ is
critical and aperiodic). More precisely, we will show the following.

%
\begin{teo} \label{TBrownianCase}
If the offspring distribution $\nu$ has finite variance $\sigma^2$,
then we have the joint convergence in distribution
\[
\biggl(\frac{\sigma}{\sqrt{n}} \mathcal{T}_n, \frac{1}{\sqrt{n}} \biggl(
\sigma+\frac{1}{\sigma} \biggr) \operatorname{Cut}_{\mathrm{v}}(\mathcal{T}_n)
\biggr) \mathop{ \longrightarrow}_{n \rightarrow\infty}\, \bigl(\mathcal{T}^{\br},
\operatorname{Cut}\bigl(\mathcal{T}^{\br}\bigr) \bigr)
\]
in $\mathbb{M} \times\mathbb{M}$.
\end{teo}

Let us explain informally why we get a factor $\sigma+1/\sigma$,
instead of the $1/\sigma$ we had in the case of the
edge-fragmentation. In the vertex-fragmentation, the average number of
deleted edges at each step is roughly $\sum_k k \nu(k) \times k =
\sigma^2 + 1$. Thus, the edge-deletions happen $\sigma^2+1$ times
faster than for the edge-fragmentation. As a consequence, $(1/\sqrt
{n}) \cdot\,\operatorname{Cut}_{\mathrm{v}}(\mathcal{T}_n)$ behaves\vspace*{1pt} approximatively like
$(1/(\sigma
^2+1) \sqrt{n}) \cdot\,\operatorname{Cut}(\mathcal{T}_n)$, that is,
$(\sigma+ 1/\sigma
)^{-1}(1/\sigma\sqrt{n}) \cdot\,\operatorname{Cut}(\mathcal{T}_n)$.

Also note that we would need additional hypotheses to extend this
result to the more general case of an offspring distribution belonging
to the domain of attraction of a Gaussian distribution. Indeed, as will
be seen in the Section~\ref{SBrownianCase}, the proof of this result
relies on the convergence of the coefficients $n/a_n^2$: if $\nu$ has
finite variance, we may and will take $a_n = \sigma\sqrt{n}$, but in
the general case, this convergence is not granted.

For both of these theorems, it is known that the first component
converges in the stronger sense of the Gromov--Hausdorff--Prokhorov
topology. However, as in the case studied by Bertoin and Miermont, the
question of whether the joint convergences hold in this sense remains open.

In the following sections, we will first work on the proof of Theorem
\ref{TMainThm}: preliminary results will be given in Section~\ref{SPrelims}, and the proof will be completed in Section~\ref{SProof}.
The global structure of this proof is close to that of \cite{BerMi},
although the technical arguments differ, especially in Section~\ref{SPrelims}. Section~\ref{SBrownianCase} will be devoted to the study
of the finite variance case.

\section{Preliminary results} \label{SPrelims}

\subsection{Modified distance on $\operatorname{Cut}_{\mathrm{v}}(\mathcal{T}_n)$}\label{SModDist}

We begin by introducing a new distance $\delta_n '$ on $\operatorname{Cut}_{\mathrm{v}}(\mathcal{T}
_n)$, defined in a similar way as the distance $\delta$ for a
continuous tree. We show that this distance is ``close'' enough to
$(a_n/n) \cdot\delta_n$, which will enable us to work on the modified
cut-tree $\operatorname{Cut}_{\mathrm{v}}'(\mathcal{T}_n):= (\operatorname{Cut}_{\mathrm{v}}(\mathcal{T}_n),\delta_n')$.

Recall the fragmentation of $\mathcal{T}_n$ introduced in Section~\ref
{SFragTn}. We now turn this process into a continuous-time
fragmentation, by saying that each vertex $v \in V(T)$ is marked
independently, with rate $\deg v / a_n$. Equivalently, this can be seen
as marking each edge of $T$ independently with rate $1/a_n$, and
deleting all the edges $e$ such that $e^- = e_i^-$ as soon as $e_i$ is
marked. Thus, we obtain a forest $\overline{\mathcal{T}}_n (t)$ at time
$t$. For
every $i \in\{1, \ldots, n\}$, we let $\mathcal{T}_{n,i} (t)$ denote the
component of $\overline{\mathcal{T}}_n (t)$ containing the edge
$e_i$, with the
convention $\mathcal{T}_{n,i}(t) = \varnothing$ if $e_i \notin
\overline{\mathcal{T}}_n(t)$,
and $\mu_{n,i} (t) = \mu_n (\mathcal{T}_{n,i} (t))$. Note that $n
\mu_{n,i}
(t)$ is the number of edges in $\mathcal{T}_{n,i} (t)$. For all $i,j
\in\{1,\ldots, n\}$, we now define
\begin{eqnarray*}
\delta_n ' (0,0) & =& 0, \qquad\delta_n
' (0, i) = \delta_n ' (i, 0) = \int
_0^{\infty} \mu_{n,i} (t) \,dt,
\\
\delta_n '  (i,j) &=& \int_{t_n (i,j)}^{\infty}
\bigl( \mu_{n,i} (t) + \mu_{n,j} (t) \bigr) \,dt,
\end{eqnarray*}
where $t_n (i,j)$ denotes the first time when the components $\mathcal{T}_{n,i}
(t)$ and $\mathcal{T}_{n,j} (t)$ become disjoint.

%
\begin{lem} \label{TModDist}
For all $i,j \in\{1,\ldots,n\}$, we have
\[
\mathbb{E} \biggl[\biggl\llvert\frac{a_n}{n} \delta_n (0,i) -
\delta_n ' (0,i)\biggr\rrvert^2 \biggr] =
\frac{a_n}{n} \mathbb{E} \bigl[\delta_n ' (0,i)
\bigr]
\]
and
\[
\mathbb{E} \biggl[\biggl\llvert\frac{a_n}{n} \delta_n (i,j) -
\delta_n ' (i,j)\biggr\rrvert^2 \biggr]
\leq\frac{a_n}{n} \mathbb{E} \bigl[\delta_n ' (0,i)
+ \delta_n ' (0,j) \bigr].
\]
\end{lem}

\begin{pf}
We work conditionally on $\mathcal{T}_n$. Fix $i \in\{1,\ldots,n\}$.
For all
$t \in\mathbb{R}_+$, we let $N_i (t)$ be the number of cuts happening
in the
component containing $e_i$ up to time $t$. Since each edge of $\mathcal{T}_n$
is marked independently with rate $1/a_n$, the process $(M_i (t))_{t
\geq0}$, where
\begin{eqnarray*}
M_i (t) &:=& \frac{a_n}{n} N_i (t) - \int
_0^t \mu_i (s) \,ds,
\end{eqnarray*}
is a purely discontinuous martingale. 
Its predictable quadratic variation can be written as
\begin{eqnarray*}
\langle M_i \rangle_t &=& \frac{a_n}{n} \int
_0^t \mu_i (s) \,ds.
\end{eqnarray*}
As a consequence, we have $\mathbb{E}[\llvert M_i (\infty)\rrvert^2]
= \mathbb{E} [ \langle M_i \rangle
_{\infty} ]$. 
Since
\begin{eqnarray*}
\lim_{t \rightarrow\infty} N_i (t) = \delta_n (0,i)
\quad\mbox{and} \quad\lim_{t \rightarrow\infty} \int_0^t
\mu_i (s) \,ds &=& \delta_n ' (0,i),
\end{eqnarray*}
we get
\begin{eqnarray*}
\mathbb{E} \biggl[\biggl\llvert\frac{a_n}{n} \delta_n (0,i) -
\delta_n ' (0,i)\biggr\rrvert^2 \biggr] &=&
\frac{a_n}{n} \mathbb{E} \bigl[\delta_n ' (0,i)
\bigr].
\end{eqnarray*}

For the second part, we use similar arguments. We fix $i \neq j \in\{
1,\ldots,n\}$, and we write $t_{ij}$ instead of $t_n (i,j)$. For all
$t \geq0$, let $\mathcal{F}_t$ denote the $\sigma$-algebra generated
by $\mathcal{T}_n$ and the atoms $\{(t_r, e_{i_r})\dvtx  t_r \leq t
\}$
of the Poisson point process of marks on the edges introduced in
Section~\ref{SFragTn}. Conditionally on $\mathcal{F}_{t_{ij}}$,
\begin{eqnarray*}
&& M_{ij} (t):= M_i (t_{ij} + t) -
M_i (t_{ij}) + M_j (t_{ij} + t) -
M_j (t_{ij})
\end{eqnarray*}
defines a purely discontinuous martingale such that
\begin{eqnarray*}
\lim_{t \rightarrow\infty} M_{ij} (t) & =& \frac{a_n}{n} \bigl(
\delta_n (b_{ij}, i ) + \delta_n
(b_{ij}, j ) \bigr) - \int_{t_{ij}}^{\infty}
\mu_i (s) \,ds - \int_{t_{ij}}^{\infty}
\mu_j (s) \,ds
\\
& =& \frac{a_n}{n} \delta_n (i,j) - \delta_n
' (i,j),
\end{eqnarray*}
where $b_{ij}$ denotes the most recent common ancestor of the leaves
$i$ and $j$ in $\operatorname{Cut}_{\mathrm{v}}(\mathcal{T}_n)$. 
Besides, since the edges of $\mathcal{T}_{n,i}$ and $\mathcal
{T}_{n,j}$ are marked
independently after time $t_{ij}$, the predictable quadratic variation
of $M_{ij}$ is
\begin{eqnarray*}
\langle M_{ij} \rangle_t &=& \frac{a_n}{n} \mathbb{E}
\biggl[\int_{t_{ij}}^{t_{ij}+t} \bigl(\mu_i (s)
+ \mu_j (s) \bigr) \,ds \biggr].
\end{eqnarray*}
Since $\delta_n ' (i,j) = \delta_n ' (0,i) + \delta_n ' (0,j) - 2
\delta_n ' (0,b_{ij})$, this yields
\begin{eqnarray*}
\mathbb{E} \biggl[\biggl\llvert\frac{a_n}{n} \delta_n (i,j) -
\delta_n ' (i,j)\biggr\rrvert^2 \biggr]
&\leq&\frac{a_n}{n} \mathbb{E} \bigl[\delta_n ' (0,i)
+ \delta_n ' (0,j) \bigr].
\end{eqnarray*}\upqed
\end{pf}

\subsection{A first joint convergence} \label{SFstJointCv}

In this section, we first state precisely the convergence theorems we
will rely on to prove the following lemmas.\vadjust{\goodbreak} To this end, we work in the
setting of sums of i.i.d. random variable $S_n = Z_1 + \cdots
+ Z_n$,
where the laws of the $Z_i$ are in the domain of attraction of a stable
law. Under additional hypotheses, Theorem~\ref{TLimLocale} below gives
a choice of scaling constants $a_n$ for which $S_n / a_n$ converges in
law to a stable variable, and a formulation of Gnedenko's local limit
theorem in this setting. Next, we will recall a result of Duquesne
which shows, in particular, the convergence (\ref{ECvTn}). The version
we will use is a joint convergence of three functions encoding the
trees $\mathcal{T}_n$ and $\mathcal{T}$. These results will allow us
to prove a first
joint convergence for the fragmented trees in Proposition~\ref{TFstJointCv}.

\subsubsection{Local limit theorem} \label{SLimLocale}

We say that a measure $\pi$ on $\mathbb{Z}$ is lattice if there exists
integers $b \in\mathbb{Z}$, $d \geq2$ such that $\operatorname
{supp}(\pi) \subset b + d
\mathbb{Z}$. We know from our hypotheses that $\nu$ is critical, aperiodic,
and $\nu(\{0\}) >0$, and these three conditions imply that $\nu$ is
nonlattice.

For\vspace*{1pt} any $\beta\in(1,2)$, we let $X^{(\beta)}$ be a stable spectrally
positive L\'{e}vy process with parameter $\beta$, and $p_t^{(\beta)}
(x)$ the density of the law of $X^{(\beta)}_t$. Similarly, for $\beta
\in(0,1)$, we let $X^{(\beta)}$ be a stable subordinator with
parameter $\beta$, and $q_t^{(\beta)} (x)$ be the density of the law
of $X^{(\beta)}_t$. We fix the normalization of these processes by
setting, for all $\lambda\geq0$,
\begin{eqnarray*}
\mathbb{E} \bigl[e^{-\lambda X^{(\beta)}_t} \bigr] &=& e^{t \lambda
^{\beta}} \qquad\mbox{if }
\beta\in(1,2),
\\
\mathbb{E} \bigl[e^{-\lambda X^{(\beta)}_t} \bigr] &=& e^{-t \lambda
^{\beta}} \qquad\mbox{if }
\beta\in(0,1).
\end{eqnarray*}
We also introduce the set $R_{\rho}$ of regularly varying functions
with index $\rho$.

%
\begin{teo} \label{TLimLocale}
Let $(Z_i, i \in\mathbb{N})$ be an i.i.d. sequence of
random variables in
$\mathbb{N}\cup\{-1,0\}$. We denote by $Z$ a random variable having
the same
law as the $Z_i$. Suppose that the law of $Z$ belongs to the domain of
attraction of a stable law of index $\beta\in(0,2) \setminus\{1\}$,
and is nonlattice. If $\beta\in(1,2)$, we also suppose that $Z$ is centered.
We introduce
\begin{eqnarray*}
S_n &=& \sum_{i=1}^n
Z_i, \qquad n \geq0.
\end{eqnarray*}
Then there exists an increasing function $A \in R_{\beta}$ and a
constant $c$ such that:
\begin{longlist}[(ii)]
\item[(i)] It holds that
%
%
\begin{equation}
\label{ELimLocale1} \mathbb{P} (Z > r ) \sim\frac{c}{A (r)}\qquad\mbox{as } r
\rightarrow\infty.
\end{equation}
\item[(ii)] Letting $a$ be the inverse function of $A$, and $a_n =
a(n)$ for all $n \in\mathbb{N}$, we have
%
%
\begin{eqnarray}
\label{ELimLocale2} \lim_{n \rightarrow\infty} \sup_{k \in\mathbb{N}}
\biggl\llvert a_n \mathbb{P} (S_n = k ) -
p_1^{(\beta)} \biggl(\frac
{k}{a_n} \biggr)\biggr\rrvert&=& 0.
\end{eqnarray}
\end{longlist}
\end{teo}

\begin{pf}
Theorem 8.3.1 of \cite{BGT} shows that, since $Z \geq-1$ a.s., the
law of $Z$ belongs to the domain of attraction of a stable law of index
$\beta$ if and only if $\mathbb{P} (Z>r ) \in R_{-\beta}$.
Using Theorem 1.5.3 of \cite{BGT}, we can take a monotone equivalent
of $\mathbb{P} (Z>r )$, hence the existence of $A$ such
that (\ref{ELimLocale1}) holds with a constant $c$ which will be
chosen hereafter.

The remarks following Theorem 8.3.1 in \cite{BGT} give a
characterization of the $a_n$ such that $S_n / a_n$ converges in law to
a stable variable of index $\beta$. In particular, it is enough to
take $a_n$ such that $n / A(a_n)$ converges, so $a = A^{-1}$ is a
suitable choice. We now choose the constant $c$ such that $S_n / a_n$
converges to $X_1^{(\beta)}$. The second point of the theorem is given
by Gnedenko's local limit theorem (see, e.g., Theorem 4.2.1 of \cite{IL}).
\end{pf}

\subsubsection{Coding the trees $\mathcal{T}_n$ and $\mathcal{T}$} \label{SCodingTrees}

We now recall three classical ways of coding a tree $T \in\mathbb
{T}$, namely the associated contour function, height function and
Lukasiewicz path. Detailed descriptions and properties of these objects
can be found, for example, in \cite{Duq}.

To define the contour function $C^{[n]}$ of $\mathcal{T}_n$, we see
$\mathcal{T}_n$ as
the embedded tree in the oriented half-plane, with each edge having
length $1$. We consider a particle that visits continuously all edges
at unit speed, from the left to the right, starting from the root.
Then, for every $t \in[0,2n]$, we let $C^{[n]}_t$ be the \emph
{height} of the particle at time~$t$, that is, its distance to the
root. The height function is defined by\vspace*{1pt} letting $H^{[n]}_j$ be the
height of the vertex $v_j$. Finally, for all $i \in\{ 0, \ldots, n \}
$, we let $Z^{[n]}_{i+1}$ be the number of offspring of the vertex
$v_i$. Then the Lukasiewicz path of $\mathcal{T}_n$ is defined~by
\begin{eqnarray*}
W^{[n]}_j &=& \sum_{i=1}^j
Z^{[n]}_i -j, \qquad j = 0, \ldots, n+1.
\end{eqnarray*}
With this definition, we have $\deg(v_j, \mathcal{T}_n) =
W^{[n]}_{j+1} -
W^{[n]}_j +1$. We extend $C^{[n]}$ and~$H^{[n]}$ by setting $C^{[n]}_t
= 0$ for all $t \in[2n, 2n+2]$ and $H^{[n]}_{n+1}=0$ (this will allow
us to keep similar scaling factors for the rescaled functions we
introduce in Theorem~\ref{TCvC,H,X}). Figure~\ref{FCodingFunctions}
gives the contour function, height function and Lukasiewicz path
associated to the tree we used in Figure~\ref{FCutT}. 

%
\begin{figure}

\includegraphics{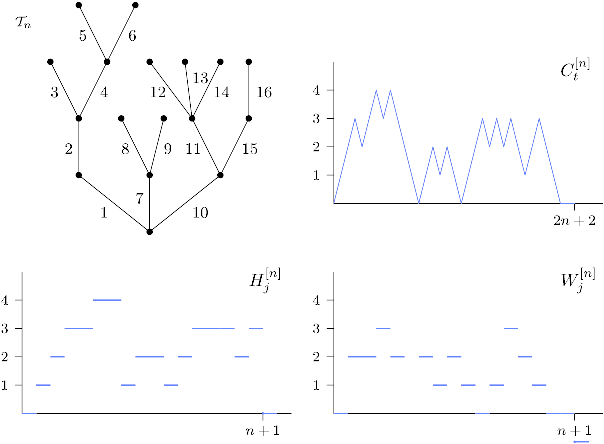}

\caption{The contour function $(C^{[n]}_t, 0 \leq t \leq2n+2)$,
height function $(H^{[n]}_j, j=0,\ldots,n+1)$ and Lukasiewicz path
$(W^{[n]}_j, j=0,\ldots,n+1)$ coding a realization of $\mathcal{T}_n$.}\label{FCodingFunctions}
\end{figure}

We also use a random walk $(W_j)_{j \geq0}$ with jump distribution
$\nu(k+1)$:
\begin{eqnarray*}
W_j &=& \sum_{i=1}^j
Z_i -j, \qquad j \geq0,
\end{eqnarray*}
where $(Z_i)_{i \in\mathbb{N}}$ are i.i.d. variables having
law $\nu$. Note
that $(W^{[n]}_j, j = 0,\ldots,n+1)$ has the same law as $(W_j, j=
0,\ldots,n+1)$ conditionally on $W_{n+1} = -1$ and \mbox{$W_{j} \geq0$} for
all $j \leq n$. In other terms, $(W_n)_{n \geq0}$ has the same law as
the Lukasiewicz path associated with a sequence of Galton--Watson trees
with offspring distribution $\nu$. From now on, we let $A$ and $a$ be
functions given by Theorem~\ref{TLimLocale} for the sequence of i.i.d. variables $(Z_i -1)_{i \in\mathbb{N}}$. Thus, we have the convergence
%
%
\begin{equation}
\label{ECvW} \frac{1}{a_n} W_n \mathop{ \longrightarrow}_{n \rightarrow\infty}^{(d)}
X^{(\alpha)}_1.
\end{equation}
Finally, let $(X_t)_{0 \leq t \leq1}$ be the excursion of length $1$
of the L\'{e}vy process $X^{(\alpha)}$, and~$(H_t)_{0 \leq t \leq1}$
be the excursion of length $1$ of the process $H^{(\alpha)}$ defined
in Section~\ref{SFragT}. We will use the following adaptation of the
results shown by Duquesne in~\cite{Duq}:

%
\begin{teo}[(Duquesne)] \label{TCvC,H,X}
Consider the rescaled functions $C^{(n)}$, $H^{(n)}$ and~$X^{(n)}$,
defined by
\begin{eqnarray*}
C^{(n)}_t &=& \frac{a_n}{n} C^{[n]}_{(2n+2)t},
\qquad H^{(n)}_t = \frac{a_n}{n} H^{[n]}_{\lfloor(n+1)t \rfloor},
\qquad X^{(n)}_t = \frac{1}{a_n} W^{[n]}_{\lfloor(n+1)t \rfloor}
\end{eqnarray*}
for all $t \in[0,1]$. If $\nu$ is aperiodic and hypothesis (\ref
{ECvW}) holds, then we have the joint convergence
\[
\bigl(C^{(n)}_t, H^{(n)}_t,
X^{(n)}_t \bigr)_{0 \leq t \leq1} \mathop{ \longrightarrow}_{n \rightarrow\infty}^{(d)}\, (H_t, H_t,
X_t )_{0 \leq t
\leq1}.
\]
\end{teo}

Proposition 4.3 of \cite{Duq} shows the convergence of the
corresponding bridges (with a change of index which comes from the fact
that we are working on trees conditioned to have $n$ edges instead of
$n$ vertices). Using the continuity of the Vervaat transform as in the
proof of \cite{Duq}, Theorem 3.1, then gives the result.

The fact that these convergences hold jointly will be used in the proof
of Lemma~\ref{TCvXtilde} below. Apart from this, we will mainly use
the convergence of the rescaled Lukasiewicz paths $X^{(n)}$, because of
the following link between the rates of our fragmentation and the jumps
of $X^{(n)}$. Recall from Section~\ref{SFragT} that $p\dvtx [0,1]
\rightarrow\mathcal{T}$ denotes the canonical projection from $[0,1]$
onto $\mathcal{T}
$. Now, the set of the branching points of $\mathcal{T}$ is $\{p(t)\dvtx  t
\in
[0,1] \mbox{ s.t. } \Delta X_t >0 \}$, and the associated local times
are $L (p(t)) = \Delta X_t$ (see \cite{DuqLG05}, proof of Theorem 4.7,
and \cite{Mi05}, Proposition 2). Similarly, we introduce the
projection $p_n$ from $K_n:= \{1/(n+1), \ldots, 1\}$ onto $V(\mathcal{T}_n)$,
such that $p_n (j/(n+1))$ is the vertex $v_{j-1}$ of $\mathcal{T}_n$.
Thus, for
all $t \in K_n$, we have
%
%
\begin{equation}
\label{ElinkX,deg} \Delta X^{(n)}_t = \frac{1}{a_n} \bigl(
\deg\bigl( p_n (t), \mathcal{T}_n\bigr) - 1\bigr).
\end{equation}

We conclude this part by showing another result of joint convergence,
for the Lukasiewicz paths of two symmetric sequences of trees. For all
$n \in\mathbb{N}$, we introduce the symmetrized tree $\widetilde
{\mathcal{T}}_n$,
obtained\vspace*{1pt} by reversing the order of the children of~each vertex of
$\mathcal{T}
_n$. We let $\widetilde{W}{}^{[n]}$ denote the Lukasiewicz path of
$\widetilde{\mathcal{T}}_n$. (We would obtain the same process by
visiting the
vertices of
$\mathcal{T}_n$ ``from right to left'' in the depth-first search.)
Finally, we
define the rescaled process $\widetilde{X}{}^{(n)}$ by
\begin{eqnarray*}
\widetilde{X}{}^{(n)}_t &=& \frac{1}{a_n}
\widetilde{W}{}^{[n]}_{\lfloor(n+1)t
\rfloor} \qquad\forall t \in[0,1].
\end{eqnarray*}

%
\begin{lem} \label{TCvXtilde}
There exists a process $(\widetilde{X}_t)_{0 \leq t \leq1}$ such that
there is the joint convergence
%
%
\begin{equation}
\label{ECvX,Xtilde} \bigl(X^{(n)}, \widetilde{X}{}^{(n)}\bigr)
\mathop{ \longrightarrow}_{n \rightarrow
\infty}^{(d)}\, (X, \widetilde{X} ).
\end{equation}
Moreover:
\begin{itemize}
\item The processes $\widetilde{X}$ and $X$ have the same law.
\item For every jump-time $t$ of $X$,
\begin{eqnarray*}
\Delta\widetilde{X}_{1-t-l(t)} &=& \Delta X_t \qquad\mbox{a.s.},
\end{eqnarray*}
where $l(t) = \inf\{ s>t\dvtx  X_s = X_{t^-}\} - t$.
\end{itemize}
\end{lem}

\begin{pf} Since $\mathcal{T}_n$ and $\widetilde{\mathcal{T}}_n$ have
the same law,
$\widetilde{X}{}^{(n)}$ converges in distribution to an excursion of the
L\'
{e}vy process $X^{(\alpha)}$ in the Skorokhod space $\mathbb{D}$.
Thus the sequence of the laws of the processes $(X^{(n)}, \widetilde
{X}{}^{(n)})$ is tight in $\mathbb{D} \times\mathbb{D}$. Up to
extraction, we can assume that $(X^{(n)}, \widetilde{X}{}^{(n)})$ converges
in distribution to a couple of processes $(X,\widetilde{X})$.

For all $n \in\mathbb{N}$, $j \in\{0,\ldots,n\}$, a simple\vspace*{1pt} computation
shows that the vertex $v_j (\mathcal{T}_n)$ corresponds to $v_{\tilde
{j}}(\widetilde{\mathcal{T}}_n)$, where
\begin{eqnarray*}
\tilde{j} &=& n - j + H^{[n]}_j - D^{[n]}_j,
\end{eqnarray*}
and $D^{[n]}_j$ is the number of strict descendants of $v_j (\mathcal{T}_n)$.
Note that $D^{[n]}_j$ is the largest integer such that $W^{[n]}_{i}
\geq W^{[n]}_j$ for all $i \in[j, j+D^{[n]}_j]$. Then (\ref
{ElinkX,deg}) shows that we have
%
%
\begin{equation}
\label{EJumpsX,Xtilde} \Delta\widetilde{X}{}^{(n)}_{(n - j + H^{[n]}_j -
D^{[n]}_j+1)/(n+1)} = \Delta
X^{(n)}_{(j+1)/(n+1)}.
\end{equation}

For all $n \in\mathbb{N}\cup\{\infty\}$, we let $(s^{(n)}_i)_{i \in
\mathbb{N}}$
be the sequence of the times where $X^{(n)}$ has a positive jump,
ranked in such a way that the sequence of the jumps $(\Delta
X^{(n)}_{s^{(n)}_i})_{i \in\mathbb{N}}$ is nonincreasing. We define the
$(\tilde{s}^{(n)}_i)_{i \in\mathbb{N}}$ in a similar way for the
$\widetilde
{X}{}^{(n)}$, $n \in\mathbb{N}\cup\{\infty\}$. Fix $i \in\mathbb
{N}$. Then (\ref{EJumpsX,Xtilde}) can be translated into
%
%
\begin{equation}
\label{ELocJumpsX,Xtilde} \tilde{s}^{(n)}_i = 1-s^{(n)}_i+
\frac{1}{n+1} \bigl(1 + H^{[n]}_{(n+1) s^{(n)}_i - 1}-D^{[n]}_{(n+1)
s^{(n)}_i - 1}
\bigr).
\end{equation}
Using the Skorokhod representation theorem, we now work under the hypothesis
\[
\bigl(H^{(n)}_t, X^{(n)}_t
\bigr)_{0 \leq t \leq1} \mathop{ \longrightarrow}_{n \rightarrow\infty}\, (H_t,
X_t )_{0
\leq t \leq1} \qquad\mbox{a.s.}
\]
Then the following convergences hold $\mbox{a.s.}$, for all $i \geq1$:
\begin{eqnarray*}
s^{(n)}_i &\displaystyle\mathop{ \longrightarrow}_{n \rightarrow\infty}&
s_i,
\\
\Delta X^{(n)}_{s^{(n)}_i} &\displaystyle\mathop{ \longrightarrow}_{n \rightarrow\infty
}&
\Delta X_{s_i},
\\
\frac{1}{n+1} H^{[n]}_{(n+1) s^{(n)}_i - 1} &\displaystyle\mathop{ \longrightarrow
}_{n \rightarrow\infty}& 0,
\\
\frac{1}{n+1} D^{[n]}_{(n+1) s^{(n)}_i - 1} &\displaystyle\mathop{ \longrightarrow
}_{n \rightarrow\infty}& l(s_i).
\end{eqnarray*}
The first two convergences hold because the $\Delta X_{s_i}$ are
distinct, and the last one uses the fact that $\mbox{a.s.}$
\[
\inf_{0 \leq u \leq\varepsilon} X_{s_i+l(s_i)+u} < X_{(s_i)^-} \qquad
\forall\varepsilon> 0.
\]
As\vspace*{1.5pt} a consequence, $\tilde{s}^{(n)}_i$ converges $\mbox{a.s.}$ to
$1-s_i-l(s_i)$. Thus, $\tilde{s}_i = 1-s_i-l(s_i)$ a.s., and
$\Delta
\widetilde{X}_{\tilde{s}_i} = \Delta X_{s_i}$ a.s. (Since the
discontinuity points are countable, this holds jointly for all $i$.)

The L\'{e}vy--It\^{o} representation theorem shows that $\widetilde{X}$
can be written as a measurable function of $(\tilde{s}_i,\Delta
\widetilde{X}_{\tilde{s}_i})_{i \in\mathbb{N}}$. This identifies
uniquely the
law of
$(X,\widetilde{X})$, hence~(\ref{ECvX,Xtilde}).
\end{pf}

\subsubsection{Joint convergence of the subtree sizes}

Recall from Section~\ref{SFragT} that $(\xi(i), i \in\mathbb{N})$
is a
sequence of i.i.d. variables in $\mathcal{T}$, with
distribution the
mass-measure $\mu$, and $\xi(0) = 0$. For all $n \in\mathbb{N}$, we
introduce independent sequences $(\xi_n (i), i \in\mathbb{N})$ of
i.i.d.
uniform integers in $\{1,\ldots, n\}$, and set $\xi_n (0) = 0$.
Recalling the notation of Section~\ref{SModDist}, we let $\tau_n
(i,j) = t_n (\xi_n (i), \xi_n (j))$ be the first time when the
components $\mathcal{T}_{n,\xi_n (i)} (t)$ and $\mathcal{T}_{n,\xi
_n (j)} (t)$ become
disjoint. Similarly, $\tau(i,j)$ will denote the first time when the
components containing $\xi(i)$ and $\xi(j)$ become disjoint in the
fragmentation of $\mathcal{T}$. Our goal is to prove the following result.

%
\begin{prop} \label{TFstJointCv}
As $n \rightarrow\infty$, we have the following weak convergences
\begin{eqnarray*}
\frac{a_n}{n} \mathcal{T}_n  &\displaystyle\mathop{ \longrightarrow} ^{(d)} &\mathcal{T},
\\
\bigl(\tau_n (i,j) \bigr)_{i,j \in\mathbb{N}}  &\displaystyle\mathop{ \longrightarrow} ^{(d)}& \bigl(\tau(i,j) \bigr)_{i,j \in\mathbb{N}},
\\
\bigl(\mu_{n,\xi_n (i)} (t) \bigr)_{i \in\mathbb{N}, t \geq0}  &\displaystyle\mathop{ \longrightarrow} ^{(d)}& \bigl(\mu_{\xi(i)} (t)
\bigr)_{i \in\mathbb{N}, t \geq0},
\end{eqnarray*}
where the three hold jointly.
\end{prop}

For the proof of this proposition, it will be convenient to identify
the $\xi_n (i)$ with vertices of $\mathcal{T}_n$ instead of edges. As
noted in
\cite{BerMi}, proof of Lemma 2, this makes no difference for the
result we seek. 

We let
\[
t^{(n)}_i = \frac{\xi_n (i)+1}{n+1},
\]
so that $p_n (t^{(n)}_i) = v_{\xi_n (i)} (\mathcal{T}_n)$.
Furthermore, we may
and will take $\xi(i) = p (t_i)$, with a sequence $(t_i, i \in\mathbb{N})$
of independent uniform variables in $[0,1]$. The sequence $(t^{(n)}_i,
i \in\mathbb{N})$ converges in distribution to $(t_i, i \in\mathbb
{N})$. Since
these sequences are independent of the trees $\mathcal{T}_n$ and
$\mathcal{T}$, the
Skorokhod representation theorem allows us to assume
%
%
\begin{equation}
\label{HasCvX,t}
\cases{ \displaystyle\bigl( X^{(n)}, \widetilde{X}{}^{(n)}
\bigr) \mathop{ \longrightarrow}_{n \rightarrow\infty}\, (X, \widetilde{X}
)\qquad\mbox{a.s.},
\vspace*{3pt}\cr
\displaystyle\bigl(t^{(n)}_i, i \in\mathbb{N} \bigr) \mathop{ \longrightarrow}_{n\rightarrow\infty}\, (t_i, i \in\mathbb{N} )\qquad
\mbox{a.s.}}
\end{equation}
We will sometimes write $X^{(\infty)}_t$ and $t^{(\infty)}_i$ for
$X_t$ and $t_i$, when it makes notation easier.

For any two vertices $u,v$ of a discrete tree $T$, we introduce the notation
\begin{eqnarray*}
[\![ u, v ]\!]_V &=& [\![ u, v ]\!] \cap V(T)
\quad\mbox{and}\quad ]\!] u, v [\![_V = [\![ u, v ]\!]_V \setminus\{u,v\},
\end{eqnarray*}
where $ [\![ u, v ]\!]$ is the segment
between $u$ and $v$ in $T$ (seen as
an $\mathbb{R}$-tree).

%
\begin{defn}
Fix $T \in\mathbb{T}$. The shape of $T$ is the discrete tree $S(T)$
such that
\begin{eqnarray*}
V\bigl(S(T)\bigr) &=& \bigl\{v \in V(T)\dvtx  \deg v \neq1 \bigr\},
\\
E\bigl(S(T)\bigr) &=& \bigl\{\{u,v\} \in V\bigl(S(T)\bigr)^2\dvtx  \forall
w \in\,]\!] u, v [\![_V, \deg w = 1\bigr\}.
\end{eqnarray*}
\end{defn}

Note that this definition can easily be extended to the case of an
$\mathbb{R}
$-tree $(T,d)$ having a finite number of leaves, by using the
``convention'' $V(T) = \{v \in T\dvtx  \deg v \neq1 \}$ in the previous definition.

For all $n,k \in\mathbb{N}$, we let $\mathcal{R}_n (k)$ denote the
shape of
the subtree of $\mathcal{T}_n$ spanned by the vertices $\xi_n (1),
\ldots,
\xi_n (k)$ and the root. Similarly, $\mathcal{R}_{\infty} (k)$ will
denote the shape of the subtree of $\mathcal{T}$ spanned by $\xi(1),
\ldots,
\xi(k)$ and the root. For all $n \in\mathbb{N}\cup\{\infty\}$, we
let $V_n
(k)$ be the set of the vertices of $\mathcal{R}_n (k)$, and we
identify the edges of $\mathcal{R}_n (k)$ with the corresponding
segments in $\mathcal{T}_n$. In particular, for any edge $e = \{u,v\}$ of
$\mathcal{R}_n (k)$, we write $w \in e$ if $w \in\,]\!] u, v [\![_V$. We
let $L_n (v)$ denote the rate at which a vertex $v$ is deleted in
$\mathcal{T}
_n$. Recall from Section~\ref{SModDist} that $L_n (v) = \deg(v,
\mathcal{T}_n)/a_n$.

%
\begin{lem} \label{TCvL}
Fix $k \in\mathbb{N}$. Under (\ref{HasCvX,t}), $\mathcal{R}_n (k)$
is $\mbox{a.s.}
$ constant for all $n$ large enough (say $n \geq N$). Identifying $V_n
(k)$ with $V_{\infty} (k)$ for all $n \geq N$, we have
\[
\bigl( L_n (v), v \in V_n (k) \bigr) \mathop{ \longrightarrow}_{n
\rightarrow\infty}\, \bigl( L (v), v \in V_{\infty} (k) \bigr)
\qquad\mbox{a.s.}
\]
\end{lem}

The above convergence can be written more rigorously by numbering the
vertices of $\mathcal{R}_n (k)$ and $\mathcal{R}_{\infty} (k)$, and
indexing on $i \in\{1, \ldots, \llvert V_{\infty} (k)\rrvert\}$, but we
keep this form to make the notation easier.

\begin{pf*}{Proof of Lemma~\ref{TCvL}}
For all $n \in\mathbb{N}\cup\{\infty\}$, $s<t \in[0,1]$, we let
\begin{eqnarray*}
I^{(n)}_{s,t} & =& \inf_{s < u < t}
X^{(n)}_u,
\end{eqnarray*}
and for all $i,j \in\mathbb{N}$,
\begin{eqnarray*}
t^{(n)}_{ij} & =& \sup\bigl\{ s \in%
\bigl[0,t^{(n)}_i \wedge t^{(n)}_j
\bigr]\dvtx  I^{(n)}_{s,t^{(n)}_i} = I^{(n)}_{s,t^{(n)}_j}\bigr
\}.
\end{eqnarray*}
Note that $p_n (t^{(n)}_{ij})$ is the most recent common ancestor of
the vertices $\xi_n (i)$ and $\xi_n (j)$ in $\mathcal{T}_n$. 
If, for example, $t^{(n)}_i < t^{(n)}_j$, we can rewrite $t^{(n)}_{ij}$ as
\[
\sup\bigl\{ s \in%
\bigl[0,t^{(n)}_i \bigr]\dvtx
X^{(n)}_{s^-} \leq I^{(n)}_{t^{(n)}_i,
t^{(n)}_j}\bigr\}.
\]
Besides, for $n=\infty$, we can replace the inequality in the broad
sense by a strict inequality:
\begin{eqnarray*}
t_{ij} &=& \sup\bigl\{ s \in%
[0,t_i ]\dvtx
X_{s^-} < I_{t_i, t_j}\bigr\}.
\end{eqnarray*}
With this notation, it is elementary to show that the following
properties hold $\mbox{a.s.}$ for all $i,j,i',j' \geq0$:
\begin{longlist}[(iii)]
\item[(i)] $X$ is continuous at $t_i$, and $X^{(n)}_{t^{(n)}_i}$
converges to $X_{t_i}$ as $n \rightarrow\infty$.
\item[(ii)] $t^{(n)}_{ij}$ converges to $t_{ij}$ as $n \rightarrow
\infty$.
\item[(iii)] $X^{(n)}_{t^{(n)}_{ij}}$ converges to $X_{t_{ij}}$ and
$X^{(n)}_{(t^{(n)}_{ij})^-}$ converges to $X_{(t_{ij})^-}$ as $n
\rightarrow\infty$.
\item[(iv)] If $t_{ij} = t_{i'j'}$, then $t^{(n)}_{ij} =
t^{(n)}_{i'j'}$ for all $n$ large enough.
\end{longlist}
We now fix $k \in\mathbb{N}$. We introduce the set
\begin{eqnarray*}
B_n (k) &=& \bigl\{t^{(n)}_{i}\dvtx  i \in\{1, \ldots,
k\}\bigr\} \cup\bigl\{t^{(n)}_{ij}\dvtx  i,j \in\{1, \ldots, k\}
\bigr\} \cup\{0\}
\end{eqnarray*}
of the times coding the vertices of $\mathcal{R}_n (k)$. We let $N_n
(k)$ be the number of elements of $B_{n} (k)$, and $b^{(n,k)}_i$ be the
$i$th element of $B_n (k)$. 
Properties (i)--(iv) can be translated into the $\mbox{a.s.}$ properties:
\begin{longlist}[(ii)$'$]
\item[(i)$'$] For $n$ large enough, $N_n (k)$ is constant.
\item[(ii)$'$] For all $i \in\{1, \ldots, N_{\infty} (k)\}$,
\begin{eqnarray*}
b^{(n,k)}_i  &\displaystyle\mathop{ \longrightarrow}_{n \rightarrow\infty}&
b^{(\infty,k)}_i,
\\
X^{(n)}_{b^{(n,k)}_i}  &\displaystyle\mathop{ \longrightarrow}_{n \rightarrow\infty
}& X_{b^{(\infty,k)}_i},
\\
X^{(n)}_{(b^{(n,k)}_i)^-}  &\displaystyle\mathop{ \longrightarrow}_{n \rightarrow
\infty}&
X_{(b^{(\infty,k)}_i)^-}.
\end{eqnarray*}
\end{longlist}
Moreover, $\mathcal{R}_n (k)$ and the $L_n (v)$, $v \in V_n (k)$, can
be recovered in a simple way using $B_n (k)$ and the $X^{(n)}_b$, $b
\in B_n (k)$:
\begin{itemize}
\item Construct a graph with vertices labeled by $B_n (k)$, the root
having label $0$.
\item For every $b \in B_n (k) \setminus\{0\}$, let $b'$ denote the
largest $b'' < b$ such that $b'' \in B_n (k)$ and $X^{(n)}_{b''} \leq
X^{(n)}_b$, then draw an edge between the vertices labelled $b$ and $b'$.
\item For each vertex $v$ labeled by $b \in B_n (k)$, let $L_n (v) =
\Delta X^{(n)}_{b} + 1/a_n$.
\end{itemize}
This entails the lemma.
\end{pf*}

This first lemma allows us to control the rate at which fragmentations
happen at the vertices of $\mathcal{R}_n (k)$. We now need another
quantity for the fragmentations happening ``on the branches'' of
$\mathcal{R}_n (k)$, that is, at vertices $v \in V (\mathcal{T}_n)
\setminus
V_n (k)$.
For every $n \in\mathbb{N}\cup\{\infty\}$, we let
\begin{eqnarray*}
\sigma_n (t) &=& \mathop{\sum_{0 < s < t}}_{X^{(n)}_{s-} <
I^{(n)}_{\st}}
\Delta X^{(n)}_s \qquad\forall t \in[0,1].
\end{eqnarray*}
If $n \in\mathbb{N}$, the quantity $a_n \sigma_n (t)$ is the sum of the
quantities $\deg v-1$ over all strict ancestors $v \neq\rho_n$ of $p_n
(t)$ in $\mathcal{T}_n$. Similarly, $\sigma(t)$ is the (infinite) sum
of the
$L(v)$ for all branching points $v$ of $\mathcal{T}$ that are on the
path $ [\![ p(t), \rho]\!]$.

%
\begin{lem} \label{TCvsigma}
With the preceding notation, in the setting of (\ref{HasCvX,t}), for
all $i \in{1, \ldots, N(k)}$, we have the convergence
\[
\sigma_n \bigl(b^{(n,k)}_i\bigr) \mathop{ \longrightarrow}_{n \rightarrow\infty
} \sigma_{\infty
} \bigl(b^{(\infty,k)}_i
\bigr)\qquad\mbox{a.s.}
\]
\end{lem}

\begin{pf}
We fix $i \in\mathbb{N}$, and let $b_n = b^{(n,k)}_i$ to simplify the
notation. For all $n \in\mathbb{N}\cup\{\infty\}$, we write $\sigma
_n (t)
= \sigma_n^- (t) + \sigma_n^+ (t)$, where
\begin{eqnarray*}
\sigma_n^+ (t) &=& \mathop{\sum_{0 < s < t}}_{X^{(n)}_{s-} <
I^{(n)}_{\st}}\bigl(X^{(n)}_{s} - I^{(n)}_{s,t} \bigr),
\\
\sigma_n^- (t) &=& \mathop{\sum_{0 < s < t}}_{X^{(n)}_{s-} <
I^{(n)}_{\st}}
\bigl(I^{(n)}_{s,t} - X^{(n)}_{s^-} \bigr).
\end{eqnarray*}
For any $s,t$ such that $0 < s < t$ and $X^{(n)}_{s-} < I^{(n)}_{\st}$,
the term $a_n (X^{(n)}_{s} - I^{(n)}_{s,t} )$ corresponds to the number
of children of $p_n(s)$ that are visited before $p_n(t)$ in the
depth-first search, and $a_n (I^{(n)}_{s,t} - X^{(n)}_{s^-} )$ is the
number of children of $p_n(s)$ that are visited after $p_n(t)$. Writing
the same decomposition $\tilde{\sigma}_n (t) = \tilde{\sigma}_n^-
(t) + \tilde{\sigma}_n^+ (t)$ for the trees $\widetilde{\mathcal
{T}}_n$, and
recalling (\ref{ELocJumpsX,Xtilde}), we thus get
\begin{eqnarray*}
\sigma_n^+ (b_n) &=& \tilde{\sigma}_n^- (
\tilde{b}_n ),
\end{eqnarray*}
where
\begin{eqnarray*}
\tilde{b}_n &=& 1-b_n+\frac{1}{n+1} \bigl(1 +
H^{[n]}_{(n+1) b_n -
1}-D^{[n]}_{(n+1) b_n - 1} \bigr).
\end{eqnarray*}

Now we note that for all $t \geq0$, we have $\sigma_n^- (t) =
X^{(n)}_{t^-}$ and $\sigma_{\infty}^- (t) = X_{t^-}$. As a
consequence, using (\ref{HasCvX,t}), we get
\[
\sigma_n^- (b_n) \mathop{ \longrightarrow}_{n \rightarrow\infty}
X_{b^-} \qquad\mbox{a.s.}
\]
The same relation for $\tilde{\sigma}_n^-$ and $\widetilde{X}{}^{(n)}$,
and the fact that $\tilde{b}_n$ converges $\mbox{a.s.}$ to $\tilde
{b}:=
1-b-l(b)$, show that
\begin{eqnarray*}
\sigma_n^+ (b_n) &=& \tilde{\sigma}_n^- (
\tilde{b}_n) \mathop{ \longrightarrow}_{n \rightarrow\infty}
\widetilde{X}_{\tilde{b}^-} \qquad\mbox{a.s.}
\end{eqnarray*}
Thus, $\sigma_n (b_n)$ converges $\mbox{a.s.}$ to $\sigma_{\infty
}^- (b) +
\tilde{\sigma}_{\infty}^- (\tilde{b})$. To show that this quantity
is equal to $\sigma_{\infty} (b)$, we introduce the ``truncated''
sums $\sigma_{n,\varepsilon} (t)$, $\sigma_{n,\varepsilon}^+ (t)$,
$\sigma_{n,\varepsilon}^- (t)$, obtained by taking into account only the
$s \in(0,t)$ such that $X^{(n)}_{s-} < I^{(n)}_{\st}$ and $\Delta
X^{(n)}_s > \varepsilon$. For all $n \in\mathbb{N}\cup\{\infty\}$, these
quantities are finite sums. Therefore, the $\mbox{a.s.}$ convergence~(\ref{HasCvX,t}) implies that for all $\varepsilon> 0$,
\[
\sigma_{\infty,\varepsilon}^+ (b) = \lim_{n \rightarrow\infty}
\sigma_{n,\varepsilon}^+ (b_n) = \lim_{n \rightarrow\infty} \tilde{
\sigma}_{n,\varepsilon}^- (\tilde{b}_n) =\tilde{\sigma}_{\infty,\varepsilon}^-
(\tilde{b}).
\]
Thus, $\sigma_{\infty,\varepsilon} (b) = \sigma_{\infty,\varepsilon}^-
(b) + \tilde{\sigma}_{\infty,\varepsilon}^- (\tilde{b})$. By letting
$\varepsilon\rightarrow0$, we get $\sigma_{\infty} (b) = \sigma
_{\infty}^- (b) + \tilde{\sigma}_{\infty}^- (\tilde{b})$.
\end{pf}

We now come back to the proof of Proposition~\ref{TFstJointCv}.

\begin{pf*}{Proof of Proposition~\ref{TFstJointCv}}
For all $n \in\mathbb{N}\cup\{\infty\}$, we add edge-lengths to the
discrete tree $\mathcal{R}_n (k)$ by letting
\begin{eqnarray*}
\ell_n \bigl(\{u,v\}\bigr) &=& d_n (u,v)\qquad\mbox{if } n
\in\mathbb{N},
\\
\ell_{\infty} \bigl(\{u,v\}\bigr) &=& d (u,v),
\end{eqnarray*}
for every edge $\{u,v\}$. Let $\mathcal{R}'_n (t)$ denote the
resulting tree with edge-lengths. We now write $\mathcal{R}_n (k,t)$
for the tree $\mathcal{R}'_n (t)$ endowed with point processes of
marks on its edges and vertices, defined as follows:
\begin{itemize}
\item The marks on the vertices of $\mathcal{R}_n (k)$ appear at the
same time as the marks on the corresponding vertices of $\mathcal{T}_n$.
\item Each edge receives a mark at its midpoint at the first time when
a vertex $v$ of $\mathcal{T}_n$ such that $v \in e$ is marked in
$\mathcal{T}_n$.
\end{itemize}
For each $n$, these two point processes are independent, and their
rates are the following:
\begin{itemize}
\item Each vertex $v \in V_n (k)$ is marked at rate $L_n (v)$,
independently of the other vertices.
\item For each edge $e$ of $\mathcal{R}_n (k)$, letting $b, b'$ denote
the points of $B_n (k)$ corresponding to $e^-, e^+$ (as explained in
the proof of Lemma~\ref{TCvL}), the edge $e$ is marked at rate $\Sigma
L_n (e)$, independently of the other edges, with
\begin{eqnarray*}
\Sigma L_n (e) & =& \sum_{v \in V (\mathcal{T}_n) \cap e}
L_n (v)
\\
& =& \sigma_n \bigl(b'\bigr) - \sigma_n (b)
+ \frac{n}{a_n^2} \bigl( H^{(n)}_{(b')^-} -
H^{(n)}_{b^-} \bigr) - L_n \bigl(e^-\bigr)
\end{eqnarray*}
if $n \in\mathbb{N}$, and
\begin{eqnarray*}
\Sigma L_{\infty} (e) &=& \Sigma L (e) = \sum_{v \in V (\mathcal{T}) \cap e}
L (v) = \sigma_{\infty} \bigl(b'\bigr) -
\sigma_{\infty} (b) - L\bigl(e^-\bigr).
\end{eqnarray*}
\end{itemize}
Now Lemmas~\ref{TCvL} and~\ref{TCvsigma} show that $L_n (v)$ and
$\Sigma L_n (e)$ converge to $L (v)$ and $\Sigma L (e)$ (resp.)
as $n \rightarrow\infty$. 
Therefore, we have the convergence
%
%
\begin{equation}
\label{ECvRmarques} \biggl( \frac{a_n}{n} \mathcal{R}_n (k,t), t
\geq0 \biggr) \mathop{ \longrightarrow}_{n \rightarrow\infty}^{(d)}\, \bigl(
\mathcal
{R}_{\infty} (k,t), t \geq0 \bigr),
\end{equation}
where $(a_n/n) \cdot\mathcal{R}_n (k,t)$ and $\mathcal{R}_{\infty}
(k,t)$ can be seen as random variables in $\mathbb{T} \times(\mathbb{R}_+
\cup\{-1\})^{\mathbb{N}} \times\{-1,0,1\}^{\mathbb{N}^2}$, for example,
\[
(a_n/n) \cdot\mathcal{R}_n (k,t) = \bigl(
\mathcal{R}_n (k), (l_i)_{i \geq1}, \bigl(
\delta_V (i,t)\bigr)_{i \geq0}, \bigl(\delta_E
(i,t)\bigr)_{i
\geq1} \bigr),
\]
where
\begin{eqnarray*}
l_i &=& \cases{ (a_n/n) \cdot\ell\bigl(e_i
\bigl(\mathcal{R}_n (k)\bigr)\bigr), &\quad if $i <
N_n(k)$,
\vspace*{2pt}\cr
-1, &\quad if $i \geq N_n (k)$,}
\\
\delta_V (i,t) &=& \cases{ 1, &\quad if $i < N_n (k)$
and the vertex $v_i \bigl(\mathcal{R}_n (k)\bigr)$
\cr
&
\qquad\quad has been marked before time $t$,
\vspace*{2pt}\cr
0, &\quad if $i < N_n (k)$
and the vertex $v_i \bigl(\mathcal{R}_n (k)\bigr)$
\cr
&
\qquad\quad has not been marked before time $t$,
\vspace*{2pt}\cr
-1,&\quad if $i \geq N_n
(k)$,}
\\
\delta_E (i,t) &=& \cases{ 1, &\quad if $i < N_n (k)$
and the edge $e_i \bigl(\mathcal{R}_n(k)\bigr)$
\cr
&\qquad\quad has been marked before time $t$,
\vspace*{2pt}\cr
0, &\quad if $i < N_n (k)$ and the edge $e_i \bigl(\mathcal{R}_n(k)\bigr)$
\cr
&\qquad\quad has not been marked before time $t$,
\vspace*{2pt}\cr
-1,&\quad if $i \geq N_n
(k)$}
\end{eqnarray*}
[recall that $N_n (k)$ is the number of vertices of $\mathcal{R}_n
(k)$]. Note that we could keep working under (\ref{HasCvX,t}) to get
an $\mbox{a.s.}$ convergence, but this is no longer necessary.

The rest of the proof goes as in \cite{BerMi}. For every $i \in
\mathbb{N}$,
we let $\eta_n (k,i,t)$ denote the number of vertices among $\xi_n
(1), \ldots, \xi_n (k)$ in the component of $\mathcal{R}_n (k)$
containing $\xi_n (i)$ at time $t$. Similarly, denote by $\eta
_{\infty} (k,i,t)$ the number of vertices among $\xi(1), \ldots, \xi
(k)$ in the component of $\mathcal{R}_{\infty} (k)$ containing $\xi
(i)$ at time $t$. It follows from (\ref{ECvRmarques}) that we have the
joint convergences
\begin{eqnarray*}
\frac{a_n}{n} \mathcal{T}_n  &\displaystyle\mathop{ \longrightarrow} ^{(d)}&
\mathcal{T},
\\
\bigl(\eta_n (k,i,t)\bigr)_{t \geq0, i \in\mathbb{N}}  &\displaystyle\mathop{\longrightarrow
} ^{(d)}& \bigl(\eta_{\infty} (k,i,t)
\bigr)_{t \geq0, i \in\mathbb{N}},
\\
\bigl(\tau_n (i,j)\bigr)_{i,j \in\mathbb{N}}  &\displaystyle\mathop{ \longrightarrow} ^{(d)}& \bigl(\tau(i,j)\bigr)_{i,j \in\mathbb{N}}.
\end{eqnarray*}
Besides, the law of large numbers gives that for each $i \in\mathbb
{N}$ and
$t \geq0$,
\[
\frac{1}{k} \eta_{\infty} (k,i,t) \mathop{ \longrightarrow}_{n
\rightarrow\infty}
\mu_{\xi(i)} (t) \qquad\mbox{a.s.}
\]
Thus, for every fixed integer $l$ and times $0 \leq t_1 \leq\cdots
\leq t_l$, we can construct a sequence $k_n \rightarrow\infty$
sufficiently slowly, such that
\[
\biggl(\frac{1}{k_n}\eta_n (k_n,i,t_j)
\biggr)_{i,j \in\{1, \ldots,l\}}  \mathop{ \longrightarrow} ^{(d)}\,
\bigl(\mu_{\xi(i)} (t_j)\bigr)_{i,j \in
\{1, \ldots,l\}},
\]
or equivalently (see \cite{AldPit}, Lemma 11)
\begin{eqnarray*}
&&\bigl( \mu_{n,\xi_n (i)} (t_j) \bigr)_{i,j \in\{1, \ldots,l\}}  \mathop{ \longrightarrow} ^{(d)}\, \bigl(\mu_{\xi(i)}
(t_j)\bigr)_{i,j \in\{1,
\ldots,l\}},
\end{eqnarray*}
both holding jointly with the preceding convergences. This entails the
proposition.
\end{pf*}

\subsection{Upper bound for the expected component mass} \label{SKeyEstimates}

To get the convergence of $(\mathcal{T}_n,\operatorname{Cut}_{\mathrm{v}}(\mathcal{T}_n))$, we
will finally need
to control the quantities
\begin{eqnarray*}
&&\mathbb{E} \biggl[\int_{2^l}^{\infty}
\mu_{n,\xi_n} (t) \,dt \biggr],
\end{eqnarray*}
where $\xi_n$ is a uniform random integer in $\{1, \ldots, n\}$. Our
main goal is to show that these quantities converge to $0$ as $l$ tends
to $\infty$, uniformly in $n$, as stated in Corollary~\ref{TCor1}.

To this end, we will sometimes work under the size-biased measure
$\GW^{\ast}$, defined as follows. We recall that a pointed tree is a
pair $(T,v)$, where $T$ is a rooted planar tree and $v$ is a vertex of
$T$. The measure $\GW^{\ast}$ is the sigma-finite measure such that,
for every pointed tree $(T,v)$,
\begin{eqnarray*}
\GW^{\ast} (T,v) &=& \mathbb{P} (\mathbf{T}=T ),
\end{eqnarray*}
where $\mathbf{T}$ is a Galton--Watson tree with offspring
distribution $\nu$. We let $\mathbb{E}^{\ast}$ denote the
expectation under this ``law.'' In particular, the conditional law
$\GW^{\ast}$ given $\llvert V (T)\rrvert= n+1$ is
well-defined, and
corresponds to the distribution of a pair $(\mathcal{T}_n,v)$ where
given $\mathcal{T}
_n$, $v$ is a uniform random vertex of $\mathcal{T}_n$. Hereafter, $T$ will
denote a \mbox{$\nu$-}Galton--Watson tree, whose expectation will either be
taken under the unbiased law or under a conditioned version of the law
$\GW^{\ast}$. Recall that we only consider values of $n$ such that $P_n
= \mathbb{P} (\llvert V (T)\rrvert= n+1 ) \neq0$.

For all $m, n \in\mathbb{N}$ such that $m \leq n$ and $P_m \neq0$,
for all
$t \in\mathbb{R}_+$, we define
%
%
\begin{equation}
\label{EDefEn} E_{m,n} (t)= \frac{1}{m} \mathbb{E} \biggl[\sum
_{e \in E (\mathcal
{T}_m)} \exp\biggl(- \sum
_{u \in[\![ \rho_m, e^-
]\!]_V} \deg(u, \mathcal{T}_m) \frac{t}{a_n}
\biggr) \biggr],
\end{equation}
and $E_n (t) = E_{n,n} (t)$. Equivalently, we can write
\begin{eqnarray*}
E_{m,n} (t)&=& \frac{1}{m} \mathbb{E}^{\ast} \biggl[\sum
_{e \in E (T)} \exp\biggl(- \sum
_{u \in[\![ \rho(T), e^-
]\!]_V} \deg(u, T) \frac{t}{a_n} \biggr) \Big| \bigl\llvert
V(T)\bigr\rrvert=m+1 \biggr].
\end{eqnarray*}
For all $m<n$, we also use the notation
\begin{eqnarray*}
&&P^{\ast}_{m,n}:= \mathbb{P}^{\ast} \bigl(\bigl\llvert
V (T_v)\bigr\rrvert= m+1 | \bigl\llvert V (T)\bigr\rrvert= n+1
\bigr),
\end{eqnarray*}
where $T_v$ denotes the tree formed by $v$ and its descendants. Our
first step is to show the following.

%
\begin{lem} \label{TMajEspMu}
Let $\xi_n$ be a uniform random edge of $\mathcal{T}_n$. Using the previous
notation, we have
%
%
\begin{equation}
\label{EmajEspMu} \mathbb{E} \bigl[\mu_{n,\xi_n} (t) \bigr] \leq
\frac{1}{n} e^{-t/a_n} + 2 \biggl(E_n (t) + \mathop{\sum
_{m=1}}_{P_m \neq0}^{n-1}
P^{\ast}_{m,n} \frac
{m}{n} E_{m,n} (t)
\biggr).
\end{equation}
\end{lem}

The proof of this lemma will use Proposition~\ref{TLoisTv} below. Let
us first introduce some notation. For all $v \in V (T)$, we let $T^v$
be the subtree obtained by deleting all the strict descendants of $v$
in $T$, and as before, $T_v$ be the tree formed by $v$ and its
descendants. We define a new tree $\hat{T}^{\hat{v}}$, constructed by
taking $T^v$ and modifying it as follows:
\begin{itemize}
\item we remove the edge $e(v)$ between $v$ and $p(v)$;
\item we add a new child $\hat{v}$ to the root, and let $\hat
{e}_{\hat{v}}$ denote the edge between $\hat{v}$ and the root;
\item we reroot the tree at $p(v)$.
\end{itemize}
An example of this construction is given in Figure~\ref
{FTransformations}. Note that we have natural bijective correspondences
between\vspace*{2pt} $V (T)$, $( V (T^v) \setminus\{v\} ) \sqcup V (T_v)$ and $( V
(\hat{T}^{\hat{v}}) \setminus\{\hat{v}\} ) \sqcup V (T_v)$, and
between $E (T)$, $E (T^v) \sqcup E (T_v)$ and $E (\hat{T}^{\hat{v}})
\sqcup E (T_v)$. Furthermore, one can easily check that for all $u \in
V (\hat{T}^{\hat{v}}) \setminus\{\hat{v}\}$, we have $\deg(u,\hat
{T}^{\hat{v}}) = \deg(u,T)$, and for all $u \in V (T_v)$, $\deg(u,T_v)
= \deg(u,T)$.

%
\begin{figure}

\includegraphics{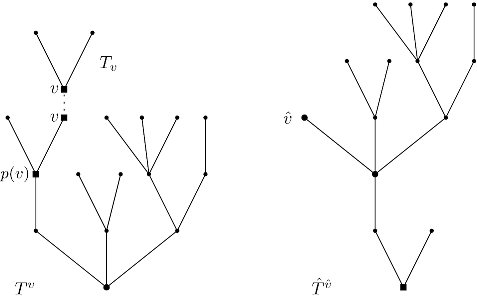}

\caption{The trees $T_v$, $T^v$ and $\hat{T}^{\hat{v}}$ obtained
from a pointed tree $(T,v)$.}\label{FTransformations}
\end{figure}

This transformation is the same as in \cite{BerMi}, page~21, except that
we work with rooted trees instead of planted trees. In our case, adding
the edge $\hat{e}_{\hat{v}}$ and deleting $e(v)$ mimics the existence
of a base edge. Thus, we can use Proposition 2 of \cite{BerMi}.

%
\begin{prop} \label{TLoisTv}
Under $\GW^{\ast}$, $(\hat{T}^{\hat{v}}, T_v)$ and $(T^v,T_v)$ have
the same ``law,'' and the trees $T^v$ and $T_v$ are independent, with
$T_v$ being a Galton--Watson tree.
\end{prop}

\begin{pf*}{Proof of Lemma~\ref{TMajEspMu}}
In this proof, we identify $\xi_n$ with the edge $e_{\xi_n}$, to make
notation easier. We first note that for each edge $e \in E (\mathcal{T}_n)$,
$e$ belongs to the component $\mathcal{T}_{n,\xi_n} (t)$ if and only
if no
vertex on the path $ [\![ e^-, \xi_n^- ]\!]_V$ has been removed at time
$t$. Given $\mathcal{T}_n$ and $\xi_n$, this happens with probability
\begin{eqnarray*}
&&\exp\biggl(- \sum_{u \in[\![ e^-, \xi_n^- ]\!]_V} \deg u \cdot
\frac
{t}{a_n} \biggr)
\end{eqnarray*}
[for any vertex $u$, at time $t$, $u$ has been deleted from the initial
tree with probability $1-\exp(- \deg u \cdot t/a_n)$]. Thus,
\begin{eqnarray*}
\mathbb{E} [n \mu_{n,\xi_n} ] &=& \mathbb{E} \biggl[\sum
_{e \in E (\mathcal{T}_n)} \mathbh{1}_{e
\in\mathcal{T}_{n,\xi_n} (t)} \biggr]
\\
&=& \mathbb{E} \biggl[
\sum_{e \in E (\mathcal{T}_n)} \exp\biggl(- \sum
_{u \in[\![ e^-, \xi_n^- ]\!]_V} \deg u \cdot\frac{t}{a_n} \biggr) \biggr].
\end{eqnarray*}
Since the edge $\xi_n$ is chosen uniformly in $E (\mathcal{T}_n)$,
this yields
\begin{eqnarray*}
\mathbb{E} [n \mu_{n,\xi_n} ] & =& \frac{1}{n} \mathbb{E} \biggl[
\sum_{e,\xi\in E (\mathcal{T}_n)} \exp\biggl(- \sum
_{u \in[\![ e^-, \xi^- ]\!]_V} \deg u \frac{t}{a_n} \biggr) \biggr]
\\
& =& \frac{1}{n} \mathbb{E} \biggl[\sum_{v \in V (\mathcal{T}_n)}
\mathbh{1}_{v \neq\rho(\mathcal{T} _n)} \sum_{e \in E (\mathcal
{T}_n)} \exp\biggl(-
\sum_{u \in[\![ e^-, p (v)
]\!]_V} \deg u \frac{t}{a_n} \biggr) \biggr],
\end{eqnarray*}
where $p (v)$ denotes the parent of vertex $v$. Hence, calling $A_n(T)$
the event $\{\llvert V (T)\rrvert= n+1\}$,
\begin{eqnarray*}
&& \mathbb{E} [n \mu_{n,\xi_n} ] = \frac{n+1}{n} \mathbb
{E}^{\ast} \biggl[\mathbh{1}_{v \neq\rho(T)} \sum
_{e \in E
(T)} \exp\biggl(- \sum_{u \in[\![
e^-, p (v) ]\!]_V}
\deg u \frac
{t}{a_n} \biggr) \Big| A_n(T) \biggr].
\end{eqnarray*}
Distinguishing the cases for which $e \in E (T_v), e \in E (T^v)
\setminus\{e(v)\}$ and $e = e(v)$, we split this quantity into three terms:
%
%
\begin{equation}
\label{EDecEspMu} \mathbb{E} [n \mu_{n,\xi_n} ] = \biggl(1+\frac
{1}{n}
\biggr) \bigl(\Sigma_v + \Sigma^v + \varepsilon_v
\bigr),
\end{equation}
where
\begin{eqnarray*}
\Sigma_v & =& \mathbb{E}^{\ast} \biggl[\mathbh{1}_{v \neq\rho(T)}
\sum_{e \in E (T_v)} \exp\biggl(- \sum
_{u \in[\![ e^-, v ]\!]_V} \bigl(\deg(u,T_v) + \deg p (v) \bigr)
\frac{t}{a_n} \biggr) \Big| A_n(T) \biggr],
\\
\Sigma^v & =& \mathbb{E}^{\ast} \biggl[\mathbh{1}_{v \neq\rho(T)}
\sum_{e \in E (T^v) \setminus\{e(v)\}} \exp\biggl(- \sum
_{u \in
[\![ e^-, p (v) ]\!]_V} \deg\bigl(u,T^v\bigr) \frac{t}{a_n}
\biggr) \Big| A_n(T) \biggr]
\end{eqnarray*}
and
\begin{eqnarray*}
\varepsilon_v &=& \mathbb{E}^{\ast} \biggl[
\mathbh{1}_{v \neq\rho
(T)} \exp\biggl(-\deg p (v) \frac{t}{a_n} \biggr) \Big|
A_n(T) \biggr].
\end{eqnarray*}

For the first term, we have
\begin{eqnarray*}
&&\Sigma_v \leq\mathbb{E}^{\ast} \biggl[\mathbh{1}_{v \neq\rho(T)}
\sum_{e \in E (T_v)} \exp\biggl(- \sum
_{u \in[\![
\rho(T_v), e^- ]\!]_V} \deg(u,T_v) \frac
{t}{a_n} \biggr) \Big|
A_n(T) \biggr].
\end{eqnarray*}
Since $\llvert V(T)\rrvert= \llvert V(T_v)\rrvert
+ \llvert V (T^v)\rrvert- 1$, this gives
\begin{eqnarray*}
\Sigma_v &\leq& \mathop{\sum_{m=1}}_{P_m \neq0}^{n-1}
P^{\ast}_{m,n} \mathbb{E}^{\ast} %
\lleft[ \sum_{e \in E (T_v)} \exp\biggl(- \sum
_{u
\in[\![ \rho(T_v), e^- ]\!]_V} \deg(u,T_v) \frac
{t}{a_n} \biggr)
\bigg|\rright.
\\
&&\hspace*{159pt}
\lleft.\begin{array} {l}
\bigl\llvert V (T_v)\bigr\rrvert= m+1,
\\[3pt]
\bigl\llvert V \bigl(T^v\bigr)\bigr\rrvert= n - m + 1
\end{array}
\rright]
\end{eqnarray*}
[$m=n$ would correspond to the case where $v = \rho(T)$, and $m=0$ to
the case where $E(T_v) = \varnothing$]. Proposition~\ref{TLoisTv} gives
that the trees $T_v$ and $T^v$ are independent, with $T_v$ being a
Galton--Watson tree. Hence,
%
%
\begin{eqnarray}\label{EMajEspMu1}
\qquad\Sigma_v & \leq&\mathop{\sum_{m=1}}_{P_m \neq0}^{n-1}
P^{\ast}_{m,n} \mathbb{E}^{\ast} \biggl[\sum
_{e \in E (T)} \exp\biggl(- \sum_{u \in
[\![ \rho(T), e^- ]\!]_V}
\deg(u,T) \frac
{t}{a_n} \biggr) \Big|
A_m(T) \biggr]
\nonumber\\[-8pt]\\[-8pt]\nonumber
& \leq&\mathop{\sum_{m=1}}_{P_m \neq0}^{n-1}
P^{\ast}_{m,n} m E_{m,n} (t).
\end{eqnarray}

For the second term, we use the correspondence between $E (T^v)
\setminus\{e(v)\}$ and $E (\hat{T}^{\hat{v}}) \setminus\{\hat
{e}_{\hat{v}}\}$, and the fact that $\rho(\hat{T}^{\hat{v}}) = p(v)$:
\begin{eqnarray*}
\Sigma^v &=& \mathbb{E}^{\ast} \biggl[\mathbh{1}_{v \neq\rho(T)}
\sum_{e \in E (\hat{T}^{\hat{v}}) \setminus\{\hat{e}_{\hat
{v}}\}} \exp\biggl(- \sum
_{u \in[\![ \rho(\hat
{T}^{\hat{v}}), e^- ]\!]_V} \deg\bigl(u,\hat{T}^{\hat{v}}\bigr)
\frac{t}{a_n} \biggr) \Big| A_n(T) \biggr].
\end{eqnarray*}
This gives
\begin{eqnarray*}
\Sigma^v &\leq&\mathbb{E}^{\ast} \biggl[\sum
_{e \in E (\hat{T}^{\hat
{v}})} \exp\biggl(- \sum_{u \in[\![ \rho(\hat{T}^{\hat{v}}),
e^- ]\!]_V}
\deg\bigl(u,\hat{T}^{\hat{v}}\bigr) \frac{t}{a_n} \biggr) \Big|
A_n(T) \biggr].
\end{eqnarray*}
Using the fact that $T^v$ and $\hat{T}^{\hat{v}}$ have the same law
under $\GW^{\ast}$, we get
\begin{eqnarray*}
&&\Sigma^v \leq\mathbb{E}^{\ast} \biggl[\sum
_{e \in E (T^v)} \exp\biggl(- \sum_{u \in[\![ \rho(T^v), e^- ]\!]_V}
\deg\bigl(u,T^v\bigr) \frac{t}{a_n} \biggr) \Big| A_n(T)
\biggr].
\end{eqnarray*}
Seeing $E(T^v)$ as a subset of $E(T)$, we can write
%
%
\begin{equation}
\label{EMajEspMu2}
\qquad\Sigma^v \leq\mathbb{E}^{\ast} \biggl[\sum
_{e \in E (T)} \exp\biggl(- \sum
_{u \in[\![ \rho(T), e^- ]\!]_V} \deg(u,T) \frac{t}{a_n} \biggr) \Big|
A_n(T) \biggr] = n E_n (t).
\end{equation}

For the third term, we simply notice that
%
%
\begin{equation}
\label{EMajEspMu3} \varepsilon_v \leq\frac{n}{n+1}
e^{-t/a_n}.
\end{equation}

Putting together (\ref{EMajEspMu1}), (\ref{EMajEspMu2}) and (\ref
{EMajEspMu3}) into (\ref{EDecEspMu}), we finally get
\begin{eqnarray*}
&&\mathbb{E} \bigl[n \mu_{n,\xi_n} (t) \bigr] \leq e^{-t/a_n} +
\biggl(1+\frac
{1}{n} \biggr) \biggl(n E_n (t) + \mathop{\sum
_{m=1}}_{P_m \neq
0}^{n-1}
P^{\ast}_{m,n} m E_{m,n} (t) \biggr).
\end{eqnarray*}
Thus,
\begin{eqnarray*}
&&\mathbb{E} \bigl[\mu_{n,\xi_n} (t) \bigr] \leq\frac{1}{n}
e^{-t/a_n} + \biggl(1+\frac
{1}{n} \biggr) \biggl(E_n (t)
+ \mathop{\sum_{m=1}}_{P_m \neq
0}^{n-1}
P^{\ast}_{m,n} \frac{m}{n} E_{m,n} (t)
\biggr).
\end{eqnarray*}\upqed
\end{pf*}

Next, we compute $E_{m,n} (t)$. To this end, we introduce two new
independent sequences of i.i.d. variables:
\begin{itemize}
\item$(\hat{Z}_i)_{i \geq1}$ with law $\hat{\nu}$, where $\hat
{\nu}$ is the size-biased version of $\nu$;
\item$(N_i)_{i \geq1}$, with same law as the number of vertices of a
Galton--Watson tree with offspring distribution $\nu$.
\end{itemize}
For all $k,h \in\mathbb{N}$, we also write
\begin{eqnarray*}
\hat{S}_h &=& \sum_{i=1}^{h}
\hat{Z}_i \quad\mbox{and} \quad Y_k = \sum
_{i=1}^k N_i.
\end{eqnarray*}

%
\begin{lem}
For every $m,n \in\mathbb{N}$ such that $m \leq n$ and $P_m \neq0$,
one has
%
%
\begin{equation}
\label{EExprEn} E_{m,n} (t) = \frac{1}{m P_m} \sum
_{1 \leq h \leq k \leq m} e^{-kt/a_n} \mathbb{P}(\hat{S}_h =
k) \mathbb{P} (Y_{k-h+1} = m-h+1 ).
\end{equation}
\end{lem}

\begin{pf}
We first note that relation (\ref{EDefEn}) can be written otherwise,
using the one-to-one correspondence $e \mapsto e^+$ between $E (T)$ and
$V (T) \setminus\{\rho(T)\}$:
\begin{eqnarray*}
E_{m,n} (t) & =& \frac{1}{m} \mathbb{E} \biggl[\sum
_{v \in V (T)
\setminus{\rho(T)}} \exp\biggl(- \sum_{u \in[\![ \rho(T), p(v) ]\!]_V}
\deg(u, T) \frac{t}{a_n} \biggr) \Big| \bigl\llvert E (T)\bigr\rrvert= m
\biggr].
\end{eqnarray*}
We thus have
\begin{eqnarray*}
E_{m,n} (t) & =& \frac{1}{m P_m} \mathbb{E} \biggl[\sum
_{v \in V (T)
\setminus{\rho(T)}} \exp\biggl(- \sum_{u \in
[\![ \rho(T), p(v) ]\!]_V}
\deg(u, T) \frac{t}{a_n} \biggr), \bigl\llvert E (T)\bigr\rrvert= m
\biggr]
\\
& =& \frac{1}{m P_m} \mathbb{E}^{\ast} %
\biggl[
\mathbh{1}_{v \neq\rho(T)} \exp\biggl(- \sum_{u \in
[\![ \rho(T), p(v) ]\!]_V}
\deg(u, T) \frac{t}{a_n} \biggr), \bigl\llvert E (T)\bigr\rrvert= m
\biggr].
\end{eqnarray*}
We now use the following description of a typical pointed tree $(T,v)$
under $\GW^{\ast}$ (see the proof of Proposition 2 of \cite{BerMi} and
\cite{LPP}):
\begin{itemize}
\item The ``law'' under $\GW^{\ast}$ of the distance $h(v)$ of the
pointed vertex $v$ to the root is the counting measure on $\mathbb
{N}\cup\{0\}$.
\item Conditionally on $h(v)=h$, the subtrees $T_v$ and $T^v$ are
independent, with $T_v$ being a Galton--Watson tree with offspring
distribution $\nu$, and $T^v$ having $\GW_h^{\ast}$ law, which can be
described as follows. $T^v$ has a distinguished branch $B = \{ u_1 =
\rho(T^v), u_2, \ldots, u_{h+1}=v \}$ of length $h$. Every vertex of
$T^v$ has an offspring that is distributed independently of the other
vertices, with offspring distribution $\nu$ for the vertices in
$V(T^v) \setminus B$, $\hat{\nu}$ for the vertices $u_1,\ldots,u_h$,
and $u_{h+1}$ having no descendants. The tree $T^v$ can thus be
constructed inductively from the root $u_1$, by choosing the $i$th
vertex $u_i$ of the distinguished branch uniformly at random from the
children of $u_{i-1}$.
\end{itemize}
In this representation, conditionally on having $h(v)=h$, $ [\![ \rho(T), p (v) ]\!]_V$ equals $\{u_1, \ldots, u_h \}$ and, for every
$i \in\{
1,\ldots,h\}$,
\[
\deg(u_i,T) = \hat{Z}_i.
\]
Besides, the total number of vertices of $T$ is the sum of the number
of vertices $h$ of $B \setminus\{v\}$, of $\llvert V(T_v)\rrvert$,
and of the\vspace*{1pt}
$\llvert V(T_u)\rrvert$ for $u$ such that $p(u) \in B
\setminus\{v\}$ and $u
\notin B$. There are $\sum_{i=1}^{h} (\hat{Z}_i - 1)$ such
trees $T_u$. Hence, under $\GW^{\ast}$:
\begin{eqnarray*}
&&\bigl\llvert E(T)\bigr\rrvert= \bigl\llvert V(T)\bigr\rrvert-1
\stackrel{(d)}
{=}Y_{\sum_{i=1}^{h}
(\hat{Z}_i - 1 ) + 1} + h - 1.
\end{eqnarray*}
Thus,
\begin{eqnarray*}
E_{m,n} (t) & =& \frac{1}{m P_m} \sum_{1 \leq h}
\mathbb{E} \Biggl[\exp\Biggl(-\sum_{i=1}^{h}
\hat{Z}_i \frac{t}{a_n} \Biggr), Y_{\sum_{i=1}^{h} \hat{Z}_i - h + 1} = m-h+1
\Biggr]
\\
& =& \frac{1}{m P_m} \sum_{1 \leq h \leq k \leq m} e^{-kt/a_n}
\mathbb{P}(\hat{S}_h = k) \mathbb{P} (Y_{k-h+1} = m-h+1 ).
\end{eqnarray*}\upqed
\end{pf}

We now compute upper bounds for the terms $\mathbb{P} (Y_{k-h+1}
= m-h+1 )$,
$\mathbb{P}(\hat{S}_h = k)$ and $(m P_m)^{-1}$.

\subsubsection*{Upper bound for $\mathbb{P}(Y_{k-h+1}=m-h+1)$}
Recalling the notation of Section~\ref{SCodingTrees}, we have
\begin{eqnarray*}
\mathbb{P} (Y_k = n ) & =& \mathbb{P} (W_n = -k \mbox{ and, }\forall p < n, W_p > -k )
\\
& =& \frac{k}{n} \mathbb{P} (W_n = -k ).
\end{eqnarray*}
The second equality is given by the cyclic lemma (see \cite{PitCSP}, Lemma~6.1). We will now use the fact, given by Theorem~\ref
{TLimLocale}, that
%
%
\begin{equation}
\label{ELimLocaleZ} \lim_{n \rightarrow\infty} \sup_{k \in\mathbb{N}}
\biggl\llvert a_n \mathbb{P} (W_n = -k ) -
p_1^{(\alpha)} \biggl(- \frac
{k}{a_n} \biggr)\biggr\rrvert=
0.
\end{equation}
For all $s,x \in(0,\infty)$, we have
\begin{eqnarray*}
x p_s^{(\alpha)} (-x) &=& s q_{x}^{(1/\alpha)}
(s)
\end{eqnarray*}
(see, e.g., \cite{BerLP}, Corollary VII.1.3). Taking $s=1$ and
$x=k/a_n$, this gives
\begin{eqnarray*}
&& \frac{k}{a_n} p_1^{(\alpha)} \biggl(-\frac{k}{a_n}
\biggr) = q_{k/a_n}^{(1/\alpha)} (1).
\end{eqnarray*}
Thus,
\begin{eqnarray*}
&& n \mathbb{P} (Y_n = k ) - q_{k/a_n}^{(1/\alpha)} (1) =
\frac{k}{a_n} \biggl(a_n \mathbb{P} (W_n = -k ) -
p_1^{(\alpha)} \biggl(-\frac
{k}{a_n} \biggr) \biggr),
\end{eqnarray*}
and we get
\begin{eqnarray*}
\mathbb{P} (Y_k = n ) & \leq&\frac{1}{n} \bigl(\bigl\llvert n
\mathbb{P} (Y_n = k ) - q_{k/a_n}^{(1/\alpha)} (1)\bigr
\rrvert+ q_{k/a_n}^{(1/\alpha)} (1) \bigr)
\\
& \leq&\frac{k}{n a_n} \biggl( \biggl\llvert a_n \mathbb{P}
(W_n = -k ) - p_1^{(\alpha)} \biggl(-
\frac{k}{a_n} \biggr)\biggr\rrvert+ p_1^{(\alpha)} \biggl(-
\frac
{k}{a_n} \biggr) \biggr).
\end{eqnarray*}
Since $p_1^{(\alpha)}$ is bounded and (\ref{ELimLocaleZ}) holds,
there exists a constant $M \in(0,\infty)$ such that, for all $k, n
\in\mathbb{N}$,
\[
\mathbb{P} (Y_k = n ) \leq\frac{k}{n a_n} M.
\]
Thus, we have the following upper bound:
%
%
\begin{equation}
\label{EmajPY} \mathbb{P} (Y_{k-h+1} = m-h+1 ) \leq\frac{k-h+1}{(m-h+1)
a_{m-h+1}} M.
\end{equation}

\subsubsection*{Upper bound for $\mathbb{P}(\hat{S}_h=k)$}
We use Theorem~\ref{TLimLocale} for the i.i.d. variables
$(\hat
{Z}_i)_{i \in\mathbb{N}}$. Let $\hat{A} \in R_{\alpha-1}$ be an increasing
function given by (i), such that
\begin{eqnarray*}
&&\mathbb{P} (\hat{Z}_1 > r ) \sim\frac{1}{\hat{A} (r)},
\end{eqnarray*}
and $\hat{a}$ be the inverse function of $\hat{A}$. Then
\begin{eqnarray*}
&&\lim_{h \rightarrow\infty} \sup_{k \in\mathbb{N}} \biggl\llvert
\hat{a}_h \mathbb{P} (\hat{S}_h = k ) -
q_1^{(\alpha-
1)} \biggl(\frac{k}{\hat{a}_h} \biggr)\biggr\rrvert=
0.
\end{eqnarray*}
Using the fact that $q_1^{(\alpha- 1)}$ is bounded, and writing
\begin{eqnarray*}
&&\mathbb{P} (\hat{S}_h = k ) \leq\frac{1}{\hat{a}_h} \biggl( \biggl
\llvert\hat{a}_h \mathbb{P} (\hat{S}_h = k ) -
q_1^{(\alpha- 1)} \biggl(\frac{k}{\hat{a}_h} \biggr)\biggr\rrvert+
q_1^{(\alpha- 1)} \biggl(\frac{k}{\hat{a}_h} \biggr) \biggr),
\end{eqnarray*}
we get the existence of a constant $M' \in(0,\infty)$ such that, for
all $h, k \in\mathbb{N}$,
%
%
\begin{equation}
\label{EmajPS} \mathbb{P} (\hat{S}_h = k ) \leq\frac{M'}{\hat{a}_h}.
\end{equation}

Furthermore, when $h$ is small enough, we have a better bound for
$\mathbb{P}(\hat{S}_h = k)$:

%
\begin{lem} \label{TDoney}
Using the previous notation, if hypothesis (\ref{HMajPZ=r}) holds,
then there exist constants $B, C$ such that for all $k \in\mathbb
{N}$, for
all $h$ such that $k / \hat{a}_h \geq B$,
\[
\mathbb{P} (\hat{S}_h = k ) \leq C \frac{h}{k \hat{A} (k)}.
\]
\end{lem}

This result is an adaptation of a theorem by Doney \cite{Don}. The
main ideas of the proof, which is rather technical, will be given in
the \hyperref[app]{Appendix}. 

Besides, using the fact that $A$ is regularly varying and an Abel
transformation of $\mathbb{P}(\hat{Z}>r)$, we get that
%
%
\begin{equation}
\label{ELienA-hatA} \frac{1}{\hat{A} (r)} \sim\frac{\alpha r}{A(r)}
\qquad\mbox{as } r
\rightarrow\infty.
\end{equation}
%

\subsubsection*{Upper bound for $(mP_m)^{-1}$}
We have
\begin{eqnarray*}
P_m &=& \mathbb{P} \bigl(\bigl\llvert E (\mathcal{T})\bigr\rrvert=
m \bigr) \sim\frac{p_1^{(\alpha)} (0)}{m a_m}
\end{eqnarray*}
(this is a straightforward consequence of the cyclic lemma and the
local limit theorem). This gives the existence of a constant $K \in
(0,\infty)$ which verifies, for all $m$ such that $P_m \neq0$,
%
%
\begin{equation}
\label{EmajPCardT} \frac{1}{m P_m} \leq K a_m.
\end{equation}

Before coming back to the proof of Corollary~\ref{TCor1}, we give
another useful result on regularly varying functions.

%
\begin{lem} \label{TLemmeFVR}
Fix $\beta\in(0,\infty)$. Let $f$ be a positive increasing function
in $R_{\beta}$ on $\mathbb{R}_+$, and $x_0$ a positive constant. For every
$\delta\in(0,\beta)$, there exists a constant $C_{\delta} \in
(0,\infty)$ such that, for all $x' \geq x \geq x_0$,
\begin{eqnarray*}
&& C_{\delta}^{-1} \biggl(\frac{x'}{x} \biggr)^{\beta- \delta}
\leq\frac{f(x')}{f(x)} \leq C_{\delta} \biggl(\frac{x'}{x}
\biggr)^{\beta+ \delta}.
\end{eqnarray*}
\end{lem}

This result is a consequence of
the Potter bounds (see, e.g., Theorem 1.5.6 of Bingham et~al. \cite{BGT}). In particular, it implies that for all $x$ bounded
away from $0$, for all $z \geq1$,
%
%
\begin{equation}
\label{EFVR1} C_{\delta}^{-1} z^{\beta- \delta} \leq
\frac{f(x z)}{f(x)} \leq C_{\delta} z^{\beta+ \delta},
\end{equation}
and likewise, for all $x \in(0,\infty)$, $z \leq1$ such that $xz$ is
bounded away from $0$,
%
%
\begin{equation}
\label{EFVR2} C_{\delta}^{-1} z^{\beta+ \delta} \leq
\frac{f(x z)}{f(x)} = \frac{f(x z)}{f (x z z^{-1})} \leq C_{\delta}
z^{\beta- \delta}.
\end{equation}

We can finally state the following.

%
\begin{lem} \label{TLimUnifIn,l}
We have
%
%
\begin{equation}
\label{ELimUnifIn,l1} \lim_{l \rightarrow\infty} \sup_{n \in\mathbb{N}}
\int
_{2^l}^{\infty} E_n (t) \,dt = 0
\end{equation}
and
\begin{eqnarray*}
&&\lim_{l \rightarrow\infty} \sup_{n \in\mathbb{N}} \mathop{
\sup_{1 \leq
m \leq n}}_{P_m \neq0} \int_{2^l}^{\infty}
\frac{m}{n} E_{m,n} (t) \,dt = 0.
\end{eqnarray*}
\end{lem}

\begin{pf}
For every $n, l \in\mathbb{N}$, we let
\begin{eqnarray*}
&& I_{n,l} = \int_{2^l}^{\infty}
E_n (t) \,dt.
\end{eqnarray*}
Putting together (\ref{EExprEn}) and (\ref{EmajPCardT}), we have
\begin{eqnarray*}
&& E_n (t) \leq K a_n \sum_{k=1}^{n}
\sum_{h=1}^{k} e^{-kt/a_n} \mathbb
{P} (\hat{S}_h = k ) \mathbb{P} (Y_{k-h+1} = n-h+1 ).
\end{eqnarray*}
This yields
\begin{eqnarray*}
I_{n,l} &\leq& K a_n^2 \sum
_{k=1}^{n} \sum_{h=1}^{k}
\frac{1}{k} e^{-2^l k/a_n} \mathbb{P} (\hat{S}_h = k )
\mathbb{P} (Y_{k-h+1} = n-h+1 ).
\end{eqnarray*}
Writing $h (n,k) = \hat{A} (k/B)\wedge\lfloor n/2 \rfloor$ and $h'
(n,k) = k
\wedge\lfloor n/2 \rfloor$, we split this sum into three parts:
\begin{eqnarray*}
I^1_{n,l} & =& a_n^2 \sum
_{k=1}^{n} \sum_{h=1}^{h (n,k)}
\frac{1}{k} e^{-2^l k/a_n} \mathbb{P} (\hat{S}_h = k )
\mathbb{P} (Y_{k-h+1} = n-h+1 ),
\\
I^2_{n,l} & =& a_n^2 \sum
_{k=1}^{n} \sum_{h=h (n,k) + 1}^{h' (n,k)}
\frac{1}{k} e^{-2^l k/a_n} \mathbb{P} (\hat{S}_h = k )
\mathbb{P} (Y_{k-h+1} = n-h+1 ),
\\
I^3_{n,l} & =& a_n^2 \sum
_{k=1}^{n} \sum_{h=h' (n,k) + 1}^{k}
\frac
{1}{k} e^{-2^l k/a_n} \mathbb{P} (\hat{S}_h = k )
\mathbb{P} (Y_{k-h+1} = n-h+1 ).
\end{eqnarray*}
Our first goal is to show that, for $i = 1, 2, 3$,
\begin{eqnarray*}
&&\lim_{l \rightarrow\infty} \sup_{n \in\mathbb{N}} I_{n,l}^i
= 0.
\end{eqnarray*}

Let us first examine $I_{n,l}^1$. Since $a$ is increasing, the upper
bound (\ref{EmajPY}) gives, for $n-h+1 \geq n/2$,
%
%
\begin{eqnarray}\label{EmajPY1-2}
\mathbb{P} (Y_{k-h+1} = n-h+1 ) & \leq& M \frac
{k-h+1}{(n-k+1)a_{n-k+1}}
\nonumber\\[-8pt]\\[-8pt]\nonumber
& \leq&2M \frac{k}{n a_{n/2}}.
\end{eqnarray}
Thus, we have
\begin{eqnarray*}
&& I_{n,l}^1 \leq2M \frac{a_n^2}{n a_{n/2}} \sum
_{k=1}^{n} e^{-2^l k /
a_n} \sum
_{h=1}^{h (n,k)} \mathbb{P} (\hat{S}_h = k
).
\end{eqnarray*}
Turning the first sum into an integral, and using the substitution $y'
= y /a_n$, we get
\begin{eqnarray*}
I_{n,l}^1 & \leq&2M \frac{a_n^2}{n a_{n/2}} \int
_1^{\infty} dy\,e^{-
2^l \lfloor y \rfloor/ a_n} \Biggl(\sum
_{h=1}^{h (n,\lfloor y
\rfloor)} \mathbb{P} \bigl(\hat{S}_h
= \lfloor y \rfloor\bigr) \Biggr)
\\
& =& 2M \frac{a_n^3}{n a_{n/2}} \int_{1/a_n}^{\infty} dy\,e^{- 2^l
\lfloor a_n y \rfloor/ a_n} \Biggl(\sum_{h=1}^{h (n,\lfloor a_n y
\rfloor)}
\mathbb{P} \bigl(\hat{S}_h = \lfloor a_n y \rfloor
\bigr) \Biggr).
\end{eqnarray*}
Since $\hat{a}$ is increasing, for all $h \leq h (n,k)$, we have $\hat
{a}_h \leq k / B$. Therefore, Lemma~\ref{TDoney} gives
\[
\mathbb{P} (\hat{S}_h = k ) \leq C \frac{h}{k \hat{A} (k)}.
\]
This yields
\begin{eqnarray*}
I_{n,l}^1 & \leq&2CM \frac{a_n^3}{n a_{n/2}} \int
_{1/a_n}^{\infty} dy\,e^{- 2^l \lfloor a_n y \rfloor/ a_n} \Biggl(\sum
_{h=1}^{h
(n,\lfloor a_n y \rfloor)} \frac{h}{a_n y \hat{A} (a_n y)} \Biggr)
\\
& \leq&2CM \frac{a_n^3}{n a_{n/2}} \int_{1/a_n}^{\infty} dy\,e^{- 2^l
\lfloor a_n y \rfloor/ a_n} \biggl(\frac{\hat{A} (\lfloor a_n y
\rfloor/B)^2}{\lfloor a_n y \rfloor\hat{A} (\lfloor a_n y \rfloor
)} \biggr).
\end{eqnarray*}
We fix $\delta\in(0,(\alpha-1) \wedge(2-\alpha))$. Since $\hat
{A}$ is regularly varying with index $\alpha- 1$, for all $y \geq1 /
a_n$, we have
\begin{eqnarray*}
&&\frac{\hat{A} (\lfloor a_n y \rfloor/B)}{\hat{A} (\lfloor a_n y
\rfloor)} \leq\frac{C_{\delta}^{-1}}{B^{\alpha-1-\delta}}
\end{eqnarray*}
[we can use (\ref{EFVR1}) because $\lfloor a_n y \rfloor/B \geq1/B$
for all
$y \in(1/a_n, \infty), n \in\mathbb{N}$]. As a~consequence, there
exists a
positive constant $K_1$ such that
\begin{eqnarray*}
I_{n,l}^1 & \leq& K_1 \frac{a_n^3}{n a_{n/2}} \int
_{1/a_n}^{\infty} dy\,e^{- 2^l \lfloor a_n y \rfloor/ a_n} \biggl(
\frac{\hat{A}
(\lfloor a_n y \rfloor)}{\lfloor a_n y \rfloor} \biggr) = K_1 J_{n,l}.
\end{eqnarray*}
Therefore, it suffices to show that
%
%
\begin{equation}
\label{ELimUnifJn,l} \lim_{l \rightarrow\infty} \sup_{n \in\mathbb{N}}
J_{n,l} = 0.
\end{equation}
To this end, we use the upper bounds~(\ref{EFVR1}) and~(\ref{EFVR2}),
with $x = a_n$ and $y = \lfloor a_n y \rfloor/ a_n$ ($x$~and~$xy$ being,
resp., greater than $a_0$ and $1$):
\begin{eqnarray*}
&&\frac{\hat{A} (\lfloor a_n y \rfloor)}{\hat{A} (a_n)} \leq C_{\delta}
\biggl( \biggl(\frac{\lfloor a_n y \rfloor
}{a_n}
\biggr)^{\alpha-1+\delta} \vee\biggl(\frac{\lfloor a_n y \rfloor
}{a_n} \biggr)^{\alpha-1-\delta}
\biggr).
\end{eqnarray*}
Thus,
\begin{eqnarray*}
&&J_{n,l} \leq\frac{a_n^2 \hat{A} (a_n)}{n a_{n/2}} \int_{1/a_n}^{\infty}
dy\,e^{- 2^l \lfloor a_n y \rfloor/ a_n} \biggl( \biggl(\frac
{a_n}{\lfloor a_n y \rfloor} \biggr)^{2-\alpha-\delta} \vee
\biggl(\frac
{a_n}{\lfloor a_n y \rfloor} \biggr)^{2-\alpha+\delta} \biggr).
\end{eqnarray*}
Using the fact that $\lfloor a_n y \rfloor\geq a_n y - 1$, and the
change of
variable $y' = y - 1/a_n$, we get
\begin{eqnarray*}
J_{n,l} & \leq&\frac{a_n^2 \hat{A} (a_n)}{n a_{n/2}} \int_0^{\infty
}dy\,e^{- 2^l y} \biggl(\frac{1}{y^{2-\alpha-\delta}} \vee\frac
{1}{y^{2-\alpha+\delta}} \biggr).
\end{eqnarray*}
Now (\ref{ELienA-hatA}) gives that $\hat{A} (a_n) / n = \hat{A}
(a_n) / A (a_n) \sim1/ \alpha a_n$, so we have
\begin{eqnarray*}
&&\frac{a_n^2 \hat{A} (a_n)}{n a_{n/2}} \sim\frac{a_n}{\alpha a_{n/2}}.
\end{eqnarray*}
Since $a$ is regularly varying with index $1/\alpha$, the right-hand
term has a finite limit as $n$ goes to infinity. Therefore, $a_n^2 \hat
{A} (a_n) / n a_{n/2}$ is bounded uniformly in $n$. Hence, there exists
a constant $K \in(0,\infty)$ such that
\begin{eqnarray*}
&&\sup_{n \in\mathbb{N}} J_{n,l} \leq K \int_0^{\infty}
dy\,e^{- 2^l y} \biggl(\frac{1}{y^{2-\alpha
-\delta}} \vee\frac{1}{y^{2-\alpha+\delta}} \biggr).
\end{eqnarray*}
This yields (\ref{ELimUnifJn,l}) by taking the limit as $l$ goes to infinity.

For the second part, we can still use (\ref{EmajPY1-2}). As in the
first step, we get
\begin{eqnarray*}
&&I_{n,l}^2 \leq2M \frac{a_n^3}{n a_{n/2}} \int
_{1/a_n}^{\infty} dy\,e^{- 2^l \lfloor a_n y \rfloor/ a_n} \Biggl(\sum
_{h=h
(n,\lfloor a_n y \rfloor) + 1}^{h' (n,\lfloor a_n y
\rfloor)} \mathbb{P} \bigl(\hat{S}_h
= \lfloor a_n y \rfloor\bigr) \Biggr).
\end{eqnarray*}
Since the sum is null if $\hat{A} (\lfloor a_n y \rfloor/ B) >
\lfloor n/2 \rfloor$, we have
\begin{eqnarray*}
&&I_{n,l}^2 \leq2M \frac{a_n^3}{n a_{n/2}} \int
_{1/a_n}^{\infty} dy\,e^{- 2^l \lfloor a_n y \rfloor/ a_n} \Biggl(\sum
_{h=\hat{A}
(\lfloor a_n y \rfloor/
B) + 1}^{\infty} \mathbb{P} \bigl(\hat{S}_h
= \lfloor a_n y \rfloor\bigr) \Biggr).
\end{eqnarray*}
We now turn the remaining sum into an integral:
\begin{eqnarray*}
&&I_{n,l}^2 \leq2M \frac{a_n^3}{n a_{n/2}} \int
_{1/a_n}^{\infty} dy\,e^{- 2^l \lfloor a_n y \rfloor/ a_n} \int
_{\hat{A} (\lfloor a_n y
\rfloor/
B)}^{\infty} dx\, \mathbb{P} \bigl(
\hat{S}_{\lfloor x+1 \rfloor} = \lfloor a_n y \rfloor\bigr).
\end{eqnarray*}
Using the change of variable $x' = \hat{A} (\lfloor a_n y \rfloor/
B) x$ and
the upper bound (\ref{EmajPS}), this gives
\begin{eqnarray*}
&&I_{n,l}^2 \leq2MM' \frac{a_n^3}{n a_{n/2}} \int
_{1/a_n}^{\infty} dy\,e^{- 2^l \lfloor a_n y \rfloor/ a_n} \int
_{1}^{\infty} dx\, \frac
{\hat{A}
(\lfloor a_n y \rfloor/ B)}{\hat{a} (\lfloor\hat{A} (\lfloor
a_n y \rfloor/ B) x + 1 \rfloor)}.
\end{eqnarray*}
Since $\hat{a}$ is increasing, for all $x, y$, we have
\begin{eqnarray*}
&&\hat{a} \bigl(\bigl\lfloor\hat{A} \bigl(\lfloor a_n y \rfloor/ B
\bigr) x + 1 \bigr\rfloor\bigr) \geq\hat{a} \bigl(\hat{A} \bigl(\lfloor
a_n y \rfloor/ B\bigr) x \bigr).
\end{eqnarray*}
Fix $\delta\in(0,1/(\alpha-1)-1)$. Inequality (\ref{EFVR1}) then
gives, for all $x \geq1$, $y \geq1/a_n$,
\begin{eqnarray*}
\hat{a} \bigl(\bigl\lfloor\hat{A} \bigl(\lfloor a_n y \rfloor/ B
\bigr) x + 1 \bigr\rfloor\bigr) & \geq& c_{\delta}^{-1} \hat{a}
\bigl(\hat{A} \bigl(\lfloor a_n y \rfloor/ B\bigr) \bigr)
x^{1/(\alpha-1)-\delta}
\\
& =& c_{\delta}^{-1} \frac{\lfloor a_n y \rfloor}{B} x^{1/(\alpha
-1)-\delta}.
\end{eqnarray*}
Thus, there exist constants $K_2, K'_2 \in(0,\infty)$ such that
\begin{eqnarray*}
&&I_{n,l}^2 \leq K_2 \frac{a_n^3}{n a_{n/2}} \int
_{1/a_n}^{\infty} dy\,e^{- 2^l
\lfloor a_n y \rfloor/ a_n} \frac{\hat{A} (\lfloor a_n y \rfloor/
B)}{\lfloor a_n y \rfloor}
\int_{1}^{\infty} \frac{dx}{x^{1/(\alpha-1)-\delta}} =
K'_2 J_{n,l},
\end{eqnarray*}
and (\ref{ELimUnifJn,l}) also gives the conclusion.

For the third part, since the terms with indices $k \leq\lfloor n/2
\rfloor$
are null, we simply use the bounds $\mathbb{P}
(Y_{k-h+1}=n-h+1 ) \leq1$ and
$\mathbb{P}(\hat{S}_h=k) \leq1$:
\begin{eqnarray*}
I_{n,l}^3 & \leq& a_n^2 \sum
_{k=\lfloor n/2 \rfloor+1}^{n} \sum
_{h=1}^{k} \frac
{1}{k} e^{-2^l k/a_n}
\\
& \leq& a_n^2 e^{-n 2^l/2 a_n} \sum
_{k=\lfloor n/2 \rfloor+1}^{n} 1
\\
& \leq& n a_n^2 e^{-n 2^l/2 a_n}.
\end{eqnarray*}
This quantity tends to $0$ as $l$ goes to infinity, uniformly in $n$.
Indeed, for any $\kappa> 0$, the function $g_{\kappa}\dvtx  x \mapsto
x^{\kappa} e^{-x}$ is bounded by a constant $G_{\kappa}$, hence
\begin{eqnarray*}
&&I_{n,l}^3 \leq G_{\kappa} \frac{2^{\kappa} a_n^{2+\kappa}}{n^{\kappa-1}}
\cdot2^{- l \kappa}.
\end{eqnarray*}
For any $\varepsilon> 0$, there exists a constant $C_{\varepsilon}$ such
that $a_n \leq C_{\varepsilon} n^{1/\alpha+ \varepsilon}$ for all $n
\in
\mathbb{N}$. Therefore, the quantity $a_n^{2+\kappa} / n^{\kappa-1}$ is
bounded as soon as $\kappa> (2+\alpha) / (\alpha- 1)$. This
completes the proof of (\ref{ELimUnifIn,l1}).

For the second limit, we note that (\ref{EExprEn}) yields
\begin{eqnarray*}
&&\int_{2^l}^{\infty} E_{m,n} (t) \,dt =
\frac{a_n}{a_m} \int_{2^l}^{\infty} E_m
(t) \,dt,
\end{eqnarray*}
for all $m \leq n$ such that $P_m \neq0$. Thus,
\begin{eqnarray*}
&&\sup_{n \in\mathbb{N}} \mathop{\sup_{1 \leq m \leq n}}_{P_m \neq
0}
\int_{2^l}^{\infty} \frac{m}{n}
E_{m,n} (t) \,dt = \sup_{n \in\mathbb{N}} \mathop{\sup
_{1 \leq m \leq n}}_{P_m \neq0} \frac{m a_n}{n a_m} I_{m,l}.
\end{eqnarray*}
As a consequence, it is enough to show that $m a_n / n a_m$ is bounded
over $\{(m,n) \in\mathbb{N}^2\dvtx  m \leq n\}$. Now,
\begin{eqnarray*}
\sup\biggl\{ \frac{m a_n}{n a_m}\dvtx  m,n \in\mathbb{N}, m\leq n \biggr\}& \leq&\sup
\biggl\{ \frac{m a_{\lambda m}}{\lambda m a_m}\dvtx  m \in\mathbb{N},
\lambda\in(1,\infty) \biggr\}
\\
& \leq&\sup\biggl\{ \frac{a_{\lambda m}}{\lambda a_m}\dvtx  m \in\mathbb{N},
\lambda\in(1,\infty)
\biggr\}.
\end{eqnarray*}
Fix $\delta\in(0,1-1/\alpha)$. Since $a$ is a positive increasing
function in $R_{1/\alpha}$, Lemma~\ref{TLemmeFVR} shows the existence
of a constant such that, for all $m \in\mathbb{N}$, $\lambda\in(1,
\infty)$,
\begin{eqnarray*}
&&\frac{a_{\lambda m}}{a_m} \leq C_{\delta} \lambda^{1/\alpha+ \delta}.
\end{eqnarray*}
Hence, for all $\lambda\in(1, \infty)$,
\begin{eqnarray*}
&&\sup_{m \in\mathbb{N}} \frac{a_{\lambda m}}{\lambda a_m} \leq C_{\delta}
\lambda^{1/\alpha+ \delta- 1} \leq C_{\delta}.
\end{eqnarray*}\upqed
\end{pf}

\subsubsection*{Key estimates for the proof of Theorem \protect\ref{TMainThm}}
We conclude this section by giving two consequences of Lemma~\ref
{TLimUnifIn,l} which will be used in the proof of Theorem~\ref{TMainThm}.

%
\begin{cor} \label{TCor1}
It holds that
\begin{eqnarray*}
&&\lim_{l \rightarrow\infty} \sup_{n \in\mathbb{N}} \mathbb{E} \biggl[
\int_{2^l}^{\infty} \mu_{n,\xi_n} (t) \,dt \biggr] =
0.
\end{eqnarray*}
\end{cor}

\begin{pf}
Using (\ref{EmajEspMu}), we get
\begin{eqnarray*}
\sup_{n \in\mathbb{N}} \mathbb{E} \biggl[\int_{2^l}^{\infty}
\mu_{n,\xi_n} (t) \,dt \biggr] &\leq&\sup_{n \in\mathbb{N}}
\frac{a_n}{n} e^{-2^l/a_n}  + 2 \sup_{n \in
\mathbb{N}} \int
_{2^l}^{\infty} E_n (t) \,dt
\\
&&{} + 2 \sup_{n \in\mathbb{N}} \sup_{1 \leq m \leq n} \int
_{2^l}^{\infty} \frac{m}{n} E_{m,n} (t)
\,dt.
\end{eqnarray*}
Lemma~\ref{TLimUnifIn,l} shows that the last two terms tend to $0$ as
$l$ goes to infinity. For the first term, we use again the fact that
for any $\kappa> 0$, the function $g_{\kappa}\dvtx  x \mapsto x^{\kappa}
e^{-x}$ is bounded by a constant $G_{\kappa}$. Hence, for all $n \in
\mathbb{N}$,
\begin{eqnarray*}
&&\frac{a_n}{n} e^{-2^l/a_n} \leq G_{\kappa} \frac{a_n^{\kappa+1}}{n}
\cdot2^{-\kappa l}.
\end{eqnarray*}
Taking $\kappa< \alpha-1$, we get that $a_n^{\kappa+1}/n$ is
bounded, which completes the proof.
\end{pf}

%
\begin{cor} \label{TCor2}
There exists a constant $C$ such that, for all $n \in\mathbb{N}$,
\begin{eqnarray*}
&&\mathbb{E} \bigl[\delta'_n (0,\xi_n)
\bigr] \leq C.
\end{eqnarray*}
\end{cor}

\begin{pf}
Recalling the definition of $\delta'_n$, we get
\begin{eqnarray*}
\mathbb{E} \bigl[\delta'_n (0,\xi_n)
\bigr] & =& \mathbb{E} \biggl[\int_0^{\infty}
\mu_{n,\xi_n} (t) \,dt \biggr].
\end{eqnarray*}
Now the upper bound (\ref{EmajEspMu}) gives
\begin{eqnarray*}
\mathbb{E} \bigl[\delta'_n (0,\xi_n)
\bigr] & \leq&1 + \mathbb{E} \biggl[\int_1^{\infty}
\mu_{n,\xi_n} (t) \,dt \biggr]
\\
& \leq&1 + \frac{a_n}{n} e^{-1/a_n} + 2 \int_{1}^{\infty}
E_n (t) \,dt + 2 \sup_{1 \leq m \leq n} \int
_{1}^{\infty} \frac{m}{n} E_{m,n} (t)
\,dt.
\end{eqnarray*}
The second term is bounded as $n \rightarrow\infty$. Recall from the
proof of Lemma~\ref{TLimUnifIn,l} that
\begin{eqnarray*}
\int_{1}^{\infty} E_n (t) \,dt &=&
I_{n,0}  \leq I_{n,0}^1 + I_{n,0}^2
+ I_{n,0}^3
\leq\bigl(K_1 + K'_2\bigr)
J_{n,0} + I_{n,0}^3.
\end{eqnarray*}
Moreover, we have seen that for any $\delta> 0$, there exists a
constant $K$ such that
\begin{eqnarray*}
&&\sup_{n \in\mathbb{N}} J_{n,0} \leq K \int_0^{\infty}
dy\,e^{-y} \biggl( \frac{1}{y^{2-\alpha-\delta}} \wedge\frac
{1}{y^{2-\alpha+\delta
}} \biggr) <
\infty,
\end{eqnarray*}
and
\begin{eqnarray*}
&&I_{n,0}^3 \leq2 n a_n^2
e^{-n/a_n}
\end{eqnarray*}
is bounded as $n \rightarrow\infty$. Since we have seen at the end of
the proof of Lemma~\ref{TMajEspMu} that there exists a constant $K'$
such that for all $n \in\mathbb{N}$, $m \leq n$ such that $P_m \neq0$,
\begin{eqnarray*}
&&\int_{1}^{\infty} \frac{m}{n}
E_{m,n} (t) \,dt \leq K' \int_{1}^{\infty}
E_m (t) \,dt,
\end{eqnarray*}
this implies the corollary.
\end{pf}

\section{Proof of Theorem \texorpdfstring{\protect\ref{TMainThm}}{1.3}} \label{SProof}

\subsection{Identity in law between $\operatorname{Cut}_{\mathrm{v}}(\mathcal{T})$ and $\mathcal{T}$}\label{SEqldeltad}

In this section, we show that the semi-infinite matrices of the mutual
distance of uniformly sampled points in $\mathcal{T}$ and
$\operatorname{Cut}_{\mathrm{v}}(\mathcal{T})$ have
the same law. This justifies the existence of $\operatorname{Cut}_{\mathrm{v}}(\mathcal{T})$, as
explained in Section~\ref{SFragT}, and shows the identity in law
between $\mathcal{T}$ and $\operatorname{Cut}_{\mathrm{v}}(\mathcal{T})$. The structure of the
proof will be
similar to that of Lemma 4 in \cite{BerMi}. Precise descriptions of
the fragmentation processes we consider can be found in \cite{Mi03}
and \cite{Mi05}.

Recall that $(\xi(i))_{i \in\mathbb{N}}$ is a sequence of i.i.d. random
variables in $\mathcal{T}$, with law $\mu$, and $\xi(0) = 0$. Since
the law
of $\mathcal{T}$ is invariant under uniform rerooting (see, e.g.,
\cite{DuqLG05}, Proposition 4.8), and the definition of $\delta$ does not
depend on the choice of the root of $\mathcal{T}$, we may assume that
$\xi(1)=
\rho$.

%
\begin{prop} \label{TEqldeltad}
It holds that
\[
\bigl( \delta\bigl(\xi(i), \xi(j)\bigr) \bigr)_{i,j \geq0} \stackrel{(d)}
{=} \bigl( d \bigl(\xi(i+1), \xi(j+1)\bigr) \bigr)_{i,j \geq0}.
\]
\end{prop}

\begin{pf}
Here, it is convenient to work on fragmentation processes taking values
in the set of the partitions of $\mathbb{N}$.

First, we introduce a process $\Pi$ which corresponds to our
fragmentation of $\mathcal{T}$ by saying that $i, j \in\mathbb{N}$
belong to the same
block of $\Pi(t)$ if and only if the path $ [\![ \xi(i), \xi(j) ]\!]_V$
does not intersect the set $\{b_k\dvtx  k \in I, t_k \leq t\}$ of the points
marked before time $t$. For every $i \in\mathbb{N}$, we let $B_i (t)$
be the
block of the partition $\Pi(t)$ containing $i$. Note that the
partitions $\Pi(t)$ are exchangeable, which justifies the existence of
the asymptotic frequencies $\lambda(B_i (t))$ of the blocks $B_i (t)$, where
\begin{eqnarray*}
\lambda(B) &=& \lim_{n \rightarrow\infty} \frac{1}{n} \bigl\llvert B
\cap\{1,\ldots,n\}\bigr\rrvert.
\end{eqnarray*}
Then we define
\begin{eqnarray*}
\sigma_i (t) &=& \inf\biggl\{ u \geq0\dvtx  \int_0^u
\lambda\bigl(B_i (s)\bigr) \,ds >t \biggr\}.
\end{eqnarray*}
We use $\sigma_i$ as a time-change, letting $\Pi' (t)$ be the
partition whose blocks are the sets $B_i (\sigma_i (t))$ for $i \in
\mathbb{N}
$. Note that this is possible because $B_i (\sigma_i (t))$ and $B_j
(\sigma_j (t))$ are either equal or disjoint.

We define a second fragmentation $\Gamma$, which results from cutting
the stable tree $\mathcal{T}$ at its heights. For every $x, y \in
\mathcal{T}$, we let
$x \wedge y$ denote the branch-point between $x$ and $y$, that is, the
unique point such that $ [\![ \rho, x \wedge y ]\!]_V = [\![ \rho, x ]\!]_V
\cap[\![ \rho, y ]\!]_V$. With this
notation, we say that $i, j \in\mathbb{N}$
belong to the same block of $\Gamma(t)$ if and only if $d (\rho,\xi
(i+1) \wedge\xi(j+1)) > t$.

Then we have the following link between the two fragmentations.

%
\begin{lem}
The fragmentation processes $\Pi'$ and $\Gamma$ have the same law.
\end{lem}

\begin{pf}
Miermont has shown in \cite{Mi05}, Theorem 1, that the process $\Pi$
is a self-similar fragmentation with index $1/\alpha$, erosion
coefficient $0$ and dislocation measure $\Delta_{\alpha}$ known
explicitly. Applying Theorem 3.3 in \cite{BerFCP}, we get that the
time-changed fragmentation $\Pi'$ is still self-similar, with index
$1/\alpha-1$, erosion coefficient $0$ and the same dislocation measure
$\Delta_{\alpha}$. Now the process $\Gamma$ is also self-similar,
with the same characteristics as $\Pi'$ (see \cite{Mi03}, Proposition 1, Theorem 1). Thus, $\Gamma$ and $\Pi'$ have the same law.
\end{pf}

Using the law of large numbers, we note that $\lambda(B_i (s)) = \mu
_{\xi(i)} (s)$ almost surely. As a consequence, $\sigma_i (t) =
\infty$ for $t = \int_0^{\infty} \lambda(B_i (s)) \,ds = \delta
(0,\xi(i))$, which means that $\delta(0,\xi(i))$ can be seen as the
first time when the singleton $\{i\}$ is a block of~$\Pi'$. Recalling
that $d (\rho,\xi(i+1)) = d (\xi(1),\xi(i+1))$ is the first time
when $\{i\}$ is a block of $\Gamma$, we get
%
%
\begin{equation}
\label{EEql1} \bigl(\delta\bigl(0,\xi(i)\bigr) \bigr)_{i \geq1}
\stackrel{(d)} {=} \bigl(d \bigl(\xi(1),\xi(i+1)\bigr) \bigr)_{i \geq1}.
\end{equation}

Similarly, for any $i \neq j \in\mathbb{N}$,
\begin{eqnarray*}
\delta\bigl(0,\xi(i) \wedge\xi(j)\bigr) & =& \frac{1}{2} \bigl(\delta
\bigl(0,\xi(i)\bigr) + \delta\bigl(0,\xi(j)\bigr) - \delta\bigl(\xi
(i), \xi(j)
\bigr)\bigr)
\\
& =& \int_0^{\tau(i,j)} \lambda\bigl(B_i
(s)\bigr) \,ds,
\end{eqnarray*}
where $\tau(i,j)$ denotes the first time when a mark appears on the
segment $ [\![ \xi(i), \xi(j) ]\!]_V$.
Thus, $\delta(0,\xi(i) \wedge
\xi(j))$ is the first time when the blocks containing $i$ and $j$ are
separated in $\Pi'$. In terms of the fragmentation $\Gamma$, this
corresponds to $d (\rho,\xi(i+1) \wedge\xi(j+1))$. Hence,
\begin{eqnarray*}
&& \bigl(\delta\bigl(0,\xi(i) \wedge\xi(j)\bigr) \bigr)_{i,j \geq1}
\stackrel{(d)} {=} \bigl(d \bigl(\xi(1),\xi(i+1) \wedge\xi(j+1)\bigr)
\bigr)_{i,j \geq1},
\end{eqnarray*}
and this holds jointly with (\ref{EEql1}). This entails the proposition.
\end{pf}

\subsection{Weak convergence}

We first establish the convergence for the cut-tree $\operatorname{Cut}_{\mathrm{v}}'(\mathcal{T}_n)$
endowed with the modified distance $\delta'_n$, as defined in
Section~\ref{SModDist}.

%
\begin{prop}
There is the joint convergence
\[
\biggl( \frac{a_n}{n} \mathcal{T}_n, \operatorname{Cut}_{\mathrm{v}}'
(\mathcal{T}_n) \biggr) \mathop{ \longrightarrow}_{n \rightarrow\infty}^{(d)}\,
\bigl(\mathcal{T}, \operatorname{Cut}_{\mathrm{v}}(\mathcal{T}) \bigr)
\]
in $\mathbb{M} \times\mathbb{M}$.
\end{prop}

\begin{pf}
Proposition~\ref{TFstJointCv} shows that for every fixed integer $l$,
there is the joint convergence
\begin{eqnarray*}
\frac{a_n}{n} \mathcal{T}_n  &\displaystyle \mathop{ \longrightarrow}_{n \rightarrow
\infty}^{(d)}&
\mathcal{T},
\\
\Biggl( 2^{-l} \sum_{j=1}^{4^l}
\mu_{n,\xi_n (i)} \bigl(j 2^{-l} \bigr) \Biggr)_{i \in\mathbb{N}}
&\displaystyle\mathop{ \longrightarrow}_{n
\rightarrow\infty}^{(d)}& \Biggl( 2^{-l} \sum
_{j=1}^{4^l} \mu_{\xi(i)} \bigl(j
2^{-l} \bigr) \Biggr)_{i \in\mathbb{N}}.
\end{eqnarray*}
Let
\begin{eqnarray*}
\Delta_{n,l} (i) &=& \mathbb{E} \Biggl[\Biggl\llvert\int
_0^{\infty} \mu_{n,\xi_n (i)} (t) \,dt -
2^{-l} \sum_{j=1}^{4^l}
\mu_{n,\xi_n (i)} \bigl(j2^{-l} \bigr)\Biggr\rrvert\Biggr].
\end{eqnarray*}
For any nonincreasing function $f\dvtx  \mathbb{R}_+ \rightarrow
[0,1 ]$, we have the upper bound
%
%
\begin{equation}
\label{EMajIntf} \Biggl\llvert\int_0^{\infty} f (t)
\,dt - 2^{-l} \sum_{j=1}^{4^l} f
\bigl(j2^{-l} \bigr)\Biggr\rrvert\leq2^{-l} + \int
_{2^l}^{\infty} f (t) \,dt.
\end{equation}
Applying this inequality to $\mu_{n,\xi_n (i)}$ yields
\begin{eqnarray*}
&&\Delta_{n,l} (i) \leq2^{-l} + \mathbb{E} \biggl[\int
_{2^l}^{\infty} \mu_{n,\xi_n} (t) \,dt \biggr].
\end{eqnarray*}
Corollary~\ref{TCor1} now shows that
\begin{eqnarray*}
&&\lim_{l \rightarrow\infty} \sup_{n \in\mathbb{N}}
\Delta_{n,l} (i) = 0,
\end{eqnarray*}
and $\Delta_{n,l} (i)$ does not depend on $i$. Besides, Proposition
\ref{TEqldeltad} shows that
\begin{eqnarray*}
&&\delta\bigl(0,\xi(i)\bigr) = \int_0^{\infty}
\mu_{\xi(i)} (t) \,dt
\end{eqnarray*}
has the same law as $d (0,\xi(i))$ and, therefore, has finite mean.
As a consequence,
\begin{eqnarray*}
&& \mathbb{E} \Biggl[\Biggl\llvert\int_0^{\infty}
\mu_{\xi(i)} (t) \,dt - 2^{-l} \sum_{j=1}^{4^l}
\mu_{\xi(i)} \bigl(j2^{-l} \bigr)\Biggr\rrvert\Biggr]
\\
&&\qquad \leq
2^{-l} + \mathbb{E} \biggl[\int_{2^l}^{\infty}
\mu_{\xi(i)} (t) \,dt \biggr]
\\
&&\qquad \mathop{ \longrightarrow}_{l \rightarrow\infty} 0,
\end{eqnarray*}
and the left-hand side does not depend on $i$. We conclude that
\[
\bigl(\delta'_n \bigl(0, \xi_n(i)\bigr)
\bigr)_{i \in\mathbb{N}} \mathop{ \longrightarrow}_{n \rightarrow\infty
}^{(d)}\, \bigl(
\delta\bigl(0, \xi(i)\bigr) \bigr)_{i \in\mathbb{N}},
\]
jointly with $(a_n/n) \cdot\mathcal{T}_n \mathop{ \longrightarrow
}\limits ^{(d)} \mathcal{T}$.

Using in addition the convergence of the $\tau_n (i,j)$ shown in
Proposition~\ref{TFstJointCv}, a~similar argument shows that the
preceding convergences also hold jointly with
\[
\bigl(\delta'_n \bigl(\xi_n (i),
\xi_n(j)\bigr) \bigr)_{i,j \in\mathbb
{N}} \mathop{ \longrightarrow}_{n \rightarrow\infty}^{(d)}\,
\bigl(\delta\bigl(\xi(i), \xi(j)\bigr) \bigr)_{i,j \in\mathbb{N}}.
\]
This entails the proposition.
\end{pf}

The convergence stated in Theorem~\ref{TMainThm} now follows
immediately. Indeed, Lemma~\ref{TModDist} and Corollary~\ref{TCor2}
show that
\[
\mathbb{E} \biggl[\biggl\llvert\frac{a_n}{n} \delta_n (i,j) -
\delta_n ' (i,j)\biggr\rrvert^2 \biggr]
\leq\frac{2 C a_n}{n}
\]
for all $i,j \geq0$ [recalling that $\xi_n (0) = 0$]. Thus, the
preceding proposition gives the joint convergence
\[
\biggl( \frac{a_n}{n} \mathcal{T}_n, \frac{a_n}{n}
\operatorname{Cut}_{\mathrm{v}}(\mathcal{T}_n) \biggr) \mathop{ \longrightarrow}_{n \rightarrow\infty}^{(d)}\, \bigl(\mathcal{T}, \operatorname{Cut}_{\mathrm{v}}(
\mathcal{T}) \bigr).
\]

\section{The finite variance case} \label{SBrownianCase}

In this section, we assume that the offspring distribution $\nu$ of
the Galton--Watson trees $\mathcal{T}_n$ has finite variance $\sigma^2$.
Theorem 23 of~\cite{AldCRT3} shows that $(\sigma/ \sqrt{n}) \cdot
\mathcal{T}
_n$ converges to the Brownian tree $\mathcal{T}^{\br}$. More precisely,
still using
the three processes described in Section~\ref{SCodingTrees} to encode
the trees $\mathcal{T}_n$, the joint convergence stated in Theorem
\ref
{TCvC,H,X} holds with $a_n = \sigma\sqrt{n}$, and limit processes
defined by $X_t = B_t$ and $H_t = 2 B_t$ for all $t \in[0,1]$. (Recall
that $B$ denotes the excursion of length 1 of the standard Brownian
motion.) Note that the normalization of $X$ is not exactly the same as
the one we used for the stable tree, since the Laplace transform of a
standard Brownian motion $B'$ is $\mathbb{E}[e^{-\lambda B'_t}] =
e^{\lambda
^2 t / 2}$. The fact that the height process $H$ is equal to $2 X$ can
be seen from the definition of $H$ as a local time, as explained in
\cite{DuqLG02}, Section~1.2.

Given these results, the proof of Theorem~\ref{TBrownianCase} follows
the same structure as that of the main theorem. We first note that the
results on the modified distance, introduced in Section~\ref
{SModDist}, still hold. In the next two sections, we will see that we
also have analogues for Proposition~\ref{TFstJointCv}, and Corollaries
\ref{TCor1} and~\ref{TCor2}.

\subsection{Convergence of the component masses}

We use the same notation as in Section~\ref{SFstJointCv}. Recall in
particular that $\mu_{n,\xi_n (i)}$ denotes the mass of the component
$\mathcal{T}_{n,\xi_n (i)} (t)$, and that $\tau_n (i,j)$ denotes the first
time when the components $\mathcal{T}_{n,\xi_n (i)} (t)$ and
$\mathcal{T}_{n,\xi_n
(j)} (t)$ become disjoint. To simplify, we drop the superscript $\br$
for the quantities associated to the Brownian tree (e.g., the
mass-measure, the mass of a component, etc.), keeping the notation we
used in the case of the stable tree. Our first step is to prove the
following result.

%
\begin{prop} \label{TJointCvBr}
As $n \rightarrow\infty$, we have the following weak convergences:
\begin{eqnarray*}
\frac{\sigma}{\sqrt{n}} \mathcal{T}_n  &\displaystyle\mathop{ \longrightarrow
} ^{(d)}& \mathcal{T}^{\br},
\\
\bigl(\tau_n (i,j) \bigr)_{i,j \geq0}  &\displaystyle\mathop{ \longrightarrow
} ^{(d)}& \biggl( \biggl(1+\frac{1}{\sigma^2}
\biggr)^{-1} \tau(i,j) \biggr)_{i,j \geq0},
\\
\bigl(\mu_{n,\xi_n (i)} (t) \bigr)_{i \geq0, t \geq0}  &\displaystyle\mathop{ \longrightarrow} ^{(d)}& \biggl(\mu_{\xi(i)} \biggl(
\biggl(1+\frac
{1}{\sigma
^2} \biggr) t \biggr) \biggr)_{i \geq0, t \geq0},
\end{eqnarray*}
where the three hold jointly.
\end{prop}

We begin by showing the same kind of property as in Lemma~\ref{TCvXtilde}.
For all $n \in\mathbb{N}$, we let $\widetilde{X}{}^{(n)}$ and $\widetilde
{C}{}^{(n)}$
denote the rescaled Lukasiewicz path and contour function of the
symmetrized tree $\widetilde{\mathcal{T}}_n$.

%
\begin{lem}
We have the joint convergence
\[
\bigl(X^{(n)}, C^{(n)}, \widetilde{X}{}^{(n)},
\widetilde{C}{}^{(n)}\bigr) \mathop{ \longrightarrow}_{n \rightarrow\infty}^{(d)}\,
(X, H, \widetilde{X}, \widetilde{H} ),
\]
where $\widetilde{H}_t = H_{1-t}$ and $\widetilde{X}_t = \widetilde
{H}_t / 2$
for all $t \in[0,1]$.
\end{lem}

\begin{pf} Since $\mathcal{T}_n$ and $\widetilde{\mathcal{T}}_n$ have
the same law,
$(\widetilde{X}{}^{(n)}, \widetilde{C}{}^{(n)})$ converges in distribution
to a
couple of processes having the same law as $(X,H)$ in $\mathbb{D}
\times\mathbb{D}$. Thus, the sequence of the laws of the processes
$(X^{(n)}, C^{(n)}, \widetilde{X}{}^{(n)}, \widetilde{C}{}^{(n)})$ is tight in
$\mathbb{D}^4$. Up to extraction, we can assume that $(X^{(n)},
C^{(n)}, \widetilde{X}{}^{(n)}, \widetilde{C}{}^{(n)})$ converges in distribution
to $(X, H,\widetilde{X}, \widetilde{H})$.

Fix\vspace*{1.5pt} $t\in[0,1]$. The definition of the contour function shows that for
all $n \in\mathbb{N}$, we have $\widetilde{C}{}^{(n)}_t = C^{(n)}_{1-t}$. Since
$H$ and $\widetilde{H}$ are $\mbox{a.s.}$ continuous, taking the limit yields
$\widetilde{H}_t = H_{1-t}$ almost surely.\vspace*{1pt} Besides, since $(X,H)$ and
$(\widetilde{X},\widetilde{H})$ have the same law, we have $\widetilde
{X}_t =
\widetilde{H}_t /2$ a.s. for all $t \in[0,1]$.

These equalities also hold $\mbox{a.s.}$, simultaneously for a countable
number of times~$t$, and the continuity of $H$, $X$, $\widetilde{H}$ and
$\widetilde{X}$ give that $\mbox{a.s.}$, they hold for all $t \in
[0,1]$. This
identifies uniquely the law of $(X,H, \widetilde{X}, \widetilde{H})$, hence
the lemma.
\end{pf}

This lemma shows that we can still work in the setting of
%
%
\begin{equation}
\label{HasCvBr} \cases{
\displaystyle \bigl( X^{(n)}, \widetilde{X}{}^{(n)}
\bigr) \mathop{ \longrightarrow}_{n\rightarrow\infty}\, (X, \widetilde{X}
)\qquad\mbox{a.s.},
\vspace*{3pt}\cr
\displaystyle\bigl(t^{(n)}_i, i \in\mathbb{N} \bigr) \mathop{ \longrightarrow}_{n \rightarrow\infty}\, (t_i, i \in\mathbb{N} )\qquad
\mbox{a.s.},}
\end{equation}
where $t^{(n)}_i = (\xi_n (i)+1)/(n+1)$ for all $n \in\mathbb{N}$,
$i \geq
0$, and $(t_i, i \in\mathbb{N})$ is a sequence of independent uniform
variables in $[0,1]$ such that $\xi(i) = p (t_i)$.

Recall the notation $\mathcal{R}_n (k)$ for the shape of the subtree
of $\mathcal{T}_n$ (or $\mathcal{T}^{\br}$ if $n = \infty$) spanned
by the root and the
vertices $\xi_n (1), \ldots, \xi_n (k)$ [or $\xi(1), \ldots, \xi
(k)$ if $n = \infty$]. We also keep the notation $L_n (v)= \deg(v,
\mathcal{T}
_n) / a_n$ for the rate at which a vertex $v$ is deleted in $\mathcal
{T}_n$ (if
$n \in\mathbb{N}$), and
\begin{eqnarray*}
\sigma_n (t) &=& \mathop{\sum_{0 < s < t}}_{X^{(n)}_{s-} <
I^{(n)}_{\st}}
\Delta X^{(n)}_s \qquad\forall t \in[0,1],
\end{eqnarray*}
where $I^{(n)}_{s,t} = \inf_{s < u < t} X^{(n)}_u$, and $X^{(\infty)}
= X$.

As in Section~\ref{SFstJointCv}, we state two lemmas which allow us
two control the rates at which the fragmentations happen on the
vertices and the edges of $\mathcal{R}_n (k)$.

%
\begin{lem} \label{TCvLBr}
Fix $k \in\mathbb{N}$. Under (\ref{HasCvBr}), $\mathcal{R}_n (k)$
is $\mbox{a.s.}$
constant for all $n$ large enough (say $n \geq N$). Identifying the
vertices of $\mathcal{R}_n (k)$ with $\mathcal{R}_{\infty} (k)$ for
all $n \geq N$, we have the $\mbox{a.s.}$ convergence
\begin{eqnarray*}
&&L_n (v) \mathop{ \longrightarrow}_{n \rightarrow\infty} 0 \qquad\forall v
\in
V \bigl(\mathcal{R}_{\infty} (k)\bigr).
\end{eqnarray*}
\end{lem}

\begin{pf}
The proof is the same as that of Lemma~\ref{TCvL}. In particular, we
get that if the $b^{(n,k)}$ are the times encoding the ``same'' vertex
$v$ of $R_n (k)$, for $n \geq N$, then we have the $\mbox{a.s.}$ convergences
\begin{eqnarray*}
b^{(n,k)}  &\displaystyle\mathop{ \longrightarrow}_{n \rightarrow\infty}& b^{(\infty,k)},
\\
X^{(n)}_{b^{(n,k)}}  &\displaystyle\mathop{ \longrightarrow}_{n \rightarrow\infty}&
X_{b^{(\infty,k)}},
\\
X^{(n)}_{(b^{(n,k)})^-}  &\displaystyle\mathop{ \longrightarrow}_{n \rightarrow\infty
}& X_{(b^{(\infty,k)})^-}.
\end{eqnarray*}
Since $X$ is now continuous, this yields
\begin{eqnarray*}
L_n (v) &=& \Delta X^{(n)}_{b^{(n,k)}} +
\frac{1}{a_n} \mathop{ \longrightarrow}_{n \rightarrow\infty} \Delta
X_{b^{(\infty,k)}} =
0.
\end{eqnarray*}\upqed
\end{pf}

%
\begin{lem} \label{TCvsigmaBr}
Let $(b_n)_{n \geq1} \in[0,1]^{\mathbb{N}}$ be a converging sequence in
$[0,1]$, and let $b$ denote its limit. Then
\begin{eqnarray*}
&&\sigma_n (b_n) \mathop{ \longrightarrow}_{n \rightarrow\infty}
H_b \qquad\mbox{a.s.}
\end{eqnarray*}
\end{lem}

\begin{pf}
As in the proof of Lemma~\ref{TCvsigma}, for all $n \in\mathbb
{N}\cup\{
\infty\}$, we write $\sigma_n (t) = \sigma_n^- (t) + \sigma_n^+
(t)$, where
\begin{eqnarray*}
&&\sigma_n^+ (t) = \mathop{\sum_{0 < s < t}}_{X^{(n)}_{s-} <
I^{(n)}_{\st}}
\bigl(X^{(n)}_{s} - I^{(n)}_{s,t} \bigr)
\quad\mbox{and} \quad\sigma_n^- (t) = \mathop{\sum
_{0 < s < t}}_{X^{(n)}_{s-} <
I^{(n)}_{\st}} \bigl(I^{(n)}_{s,t}
- X^{(n)}_{s^-} \bigr).
\end{eqnarray*}
For all $t \geq0$, $n \in\mathbb{N}$, we have $\sigma_n^- (t) =
X^{(n)}_{t^-}$. As a consequence, (\ref{HasCvBr}) gives
\begin{eqnarray*}
&&\sigma_n^- (b_n) \mathop{ \longrightarrow}_{n \rightarrow\infty}
X_b \qquad\mbox{a.s.}
\end{eqnarray*}
Besides, we still have $\sigma_n^+ (b_n) = \tilde{\sigma}_n^- (
\tilde{b}_n )$, with
\begin{eqnarray*}
&&\tilde{b}_n = 1-b_n+\frac{1}{n+1} \bigl(1 +
H^{[n]}_{(n+1) b_n -
1}-D^{[n]}_{(n+1) b_n - 1} \bigr).
\end{eqnarray*}
Now
\begin{eqnarray*}
&&\tilde{b}_n \mathop{ \longrightarrow}_{n \rightarrow\infty} 1 - b - l(b),
\end{eqnarray*}
where $l(b) = \inf\{ s>b\dvtx  X_s = X_{b}\} - b$. Using (\ref{HasCvBr})
again, we get
\begin{eqnarray*}
&&\sigma_n^+ (b_n) \mathop{ \longrightarrow}_{n \rightarrow\infty}
\widetilde{X}_{1-b-l(b)} = X_{b+l(b)} = X_b \qquad
\mbox{a.s.}
\end{eqnarray*}
Thus, we have the $\mbox{a.s.}$ convergence
\begin{eqnarray*}
&&\sigma_n (b_n) \mathop{ \longrightarrow}_{n \rightarrow\infty} 2
X_b = H_b.
\end{eqnarray*}\upqed
\end{pf}

We can now give the proof of Proposition~\ref{TJointCvBr}.

\begin{pf*}{Proof of Proposition~\ref{TJointCvBr}}
Fix $n \in\mathbb{N}\cup\{\infty\}$. As in the proof of Proposition~\ref{TFstJointCv}, we write $\mathcal{R}_n (k,t)$ for the reduced tree
with edge-lengths, endowed with point processes of marks on its edges
and vertices such that:
\begin{itemize}
\item The marks on the vertices of $\mathcal{R}_n (k)$ appear at the
same time as the marks on the corresponding vertices of $\mathcal{T}_n$.
\item Each edge receives a mark at its midpoint at the first time when
a vertex $v$ of $\mathcal{T}_n$ such that $v \in e$ is marked in
$\mathcal{T}_n$.
\end{itemize}
These two point processes are independent, and their rates are the following:
\begin{itemize}
\item If $n \in\mathbb{N}$, each vertex $v$ of $\mathcal{R}_n (k)$
is marked
at rate $L_n (v)$, independently of the other vertices. If $n = \infty
$, there are no marks on the vertices.
\item For each edge $e$ of $\mathcal{R}_n (k)$, letting $b, b'$ denote
the points of $B_n (k)$ corresponding to $e^-, e^+$, the edge $e$ is
marked at rate $\Sigma L_n (e)$, independently of the other edges, with
\begin{eqnarray*}
\Sigma L_n (e) & =& \sum_{v \in V (\mathcal{T}_n) \cap e}
L_n (v)
\\
& =& \sigma_n \bigl(b'\bigr) - \sigma_n (b)
+ \frac{n}{a_n^2} \bigl( H^{(n)}_{(b')^-} -
H^{(n)}_{b^-} \bigr) - L_n \bigl(e^-\bigr)
\end{eqnarray*}
if $n \in\mathbb{N}$, and
\begin{eqnarray*}
&&\Sigma L_{\infty} (e) = H_{b'} - H_b.
\end{eqnarray*}
\end{itemize}
We see from Lemmas~\ref{TCvLBr} and~\ref{TCvsigmaBr} that $L_n (v)$
converges to $0$ as $n \rightarrow\infty$, and that
\begin{eqnarray*}
&&\Sigma L_n (e) \mathop{ \longrightarrow}_{n \rightarrow\infty}\, \biggl(1 +
\frac
{1}{\sigma^2} \biggr) \Sigma L_{\infty} (e).
\end{eqnarray*}
As a consequence, we have the convergence
%
%
\begin{equation}
\label{ECvRBr} \biggl( \frac{a_n}{n} \mathcal{R}_n (k,t), t
\geq0 \biggr) \mathop{ \longrightarrow}_{n \rightarrow\infty}^{(d)}\, \biggl(
\mathcal{R}_{\infty} \biggl(k, \biggl(1+\frac{1}{\sigma^2} \biggr)t
\biggr), t \geq0 \biggr).
\end{equation}
[As in the case $\alpha\in(1,2)$, $(a_n/n) \cdot\mathcal{R}_n
(k,t)$ and $\mathcal{R}_{\infty} (k,t)$ can be seen as random
variables in $\mathbb{T} \times(\mathbb{R}_+ \cup\{-1\})^{\mathbb
{N}} \times\{
-1,0,1\}^{\mathbb{N}^2}$.]

For all $i \in\mathbb{N}$, we let $\eta_n (k,i,t)$ denote the number of
vertices among $\xi_n (1), \ldots,  \xi_n (k)$ in the component of
$\mathcal{R}_n (k)$ containing $\xi_n (i)$ at time $t$, and similarly
$\eta_{\infty} (k,i,t)$ the number of vertices among $\xi(1), \ldots,
\xi(k)$ in the component of~$\mathcal{R}_{\infty} (k)$ containing
$\xi(i)$ at time $t$. It follows from (\ref{ECvRBr}) that we have the
joint convergences
\begin{eqnarray*}
\frac{a_n}{n} \mathcal{T}_n &\displaystyle \mathop{ \longrightarrow} ^{(d)}&
\mathcal{T}^{\br},
\\
\bigl(\eta_n (k,i,t)\bigr)_{t \geq0, i \in\mathbb{N}} &\displaystyle \mathop{\longrightarrow
} ^{(d)}& \biggl(\eta_{\infty} \biggl(k,i,
\biggl(1+\frac{1}{\sigma^2} \biggr)t \biggr) \biggr)_{t \geq0, i \in
\mathbb{N}},
\\
\bigl(\tau_n (i,j)\bigr)_{i,j \in\mathbb{N}} &\displaystyle \mathop{ \longrightarrow} ^{(d)}& \biggl( \biggl(1+\frac
{1}{\sigma^2}
\biggr)^{-1}\tau(i,j) \biggr)_{i,j \in\mathbb{N}}.
\end{eqnarray*}
The end of the proof is the same as for Proposition~\ref{TFstJointCv}.
\end{pf*}

\subsection{Upper bound for the expected component mass}

The second step is to show that, as in Section~\ref{SKeyEstimates},
the following properties hold.

%
\begin{lem} \label{TKEBr}
It holds that
\begin{eqnarray*}
&&\lim_{l \rightarrow\infty} \sup_{n \in\mathbb{N}} \mathbb{E} \biggl[
\int_{2^l}^{\infty} \mu_{n,\xi_n} (t) \,dt \biggr] =
0.
\end{eqnarray*}
Besides, there exists a constant $C$ such that, for all $n \in\mathbb{N}$,
\begin{eqnarray*}
&&\mathbb{E} \bigl[\delta'_n (0,\xi_n)
\bigr] \leq C.
\end{eqnarray*}
\end{lem}

\begin{pf}
We use the fact that there exists a natural coupling between the
edge-fragmentation and the vertex-fragmentation of $\mathcal{T}_n$. Indeed,
both can be obtained by a deterministic procedure, given $\mathcal
{T}_n$ and a
uniform permutation $(i_1,\ldots,i_n)$ of $\{1,\ldots,n\}$. More
precisely, in the edge-fragmentation, we delete the edge $e_{i_k}$ at
each step $k$, thus splitting $\mathcal{T}_n$ into at most two connected
components, whereas in the vertex fragmentation, we delete all the
edges such that $e^- = e_{i_k}^-$. Thus, at each step, the connected
component containing a given edge $e$ for the vertex-fragmentation is
included in the component containing $e$ for the edge-fragmentation.

Now consider the continuous-time versions of these fragmentations: each
edge is marked independently with rate $a_n/n = \sigma/\sqrt{n}$ in
our case, and $1/\sqrt{n}$ in \cite{BerMi}. We let $\mathcal{T}_{n,i}^{E}
(t)$ and $\mathcal{T}_{n,i}^{V} (t)$ denote the connected components containing
the edge $e_i$ at time $t$, respectively, for the edge-fragmentation
and the vertex-fragmentation. Then the preceding remark shows that
there exists a coupling such that $\mathcal{T}_{n,i}^{V} (t) \subset
\mathcal{T}
_{n,i}^{E} (\sigma t)$ a.s., and thus $\mu_n (\mathcal
{T}_{n,i}^{V} (t))
\leq\mu_n (\mathcal{T}_{n,i}^{E} (\sigma t))$ almost surely.

Lemma 3 and Corollary 1 of \cite{BerMi} show that the two announced
properties hold for the case of the edge-fragmentation. Therefore, they
also hold for the vertex-fragmentation.
\end{pf}

\subsection{Proof of Theorem \texorpdfstring{\protect\ref{TBrownianCase}}{1.4}}

As before, the proof of Theorem~\ref{TBrownianCase} now relies on
showing a joint convergence for the rescaled versions of $\mathcal
{T}_n$ and
the modified cut-tree $\operatorname{Cut}_{\mathrm{v}}(\mathcal{T}_n)$:
%
%
\begin{equation}
\label{EJointCvCut} \biggl( \frac{a_n}{n} \mathcal{T}_n, \biggl(1+
\frac{1}{\sigma
^2} \biggr) \operatorname{Cut}_{\mathrm{v}}' (\mathcal{T}_n)
\biggr) \mathop{ \longrightarrow}_{n \rightarrow
\infty}^{(d)}\, \bigl(
\mathcal{T}^{\br}, \operatorname{Cut}\bigl(\mathcal{T}^{\br}
\bigr) \bigr)
\end{equation}
in $\mathbb{M} \times\mathbb{M}$. Indeed, Lemma~\ref{TModDist} and
the second part of Lemma~\ref{TKEBr} show that
\begin{eqnarray*}
&&\mathbb{E} \biggl[\biggl\llvert\frac{a_n}{n} \delta_n (i,j) -
\delta_n ' (i,j)\biggr\rrvert^2 \biggr]
\leq\frac{2 C a_n}{n}
\end{eqnarray*}
for all $i,j \geq0$. Thus, (\ref{EJointCvCut}) entails the joint convergence
\begin{eqnarray*}
&&\biggl( \frac{a_n}{n} \mathcal{T}_n, \frac{a_n}{n} \biggl(1+
\frac
{1}{\sigma
^2} \biggr) \operatorname{Cut}_{\mathrm{v}}(\mathcal{T}_n) \biggr)
\mathop{ \longrightarrow}_{n
\rightarrow\infty}^{(d)}\, \bigl(\mathcal{T},
\operatorname{Cut}_{\mathrm{v}}(\mathcal{T}) \bigr).
\end{eqnarray*}
Since $a_n = \sigma\sqrt{n}$, this gives Theorem~\ref{TBrownianCase}.

Let us finally justify why (\ref{EJointCvCut}) holds. Proposition~\ref
{TJointCvBr} shows that for every fixed integer $l$, there is the joint
convergence
\begin{eqnarray*}
\frac{a_n}{n} \mathcal{T}_n &\displaystyle \mathop{ \longrightarrow}_{n \rightarrow
\infty}^{(d)}&
\mathcal{T}^{\br},
\\
\Biggl( 2^{-l} \sum_{j=1}^{4^l}
\mu_{n,\xi_n (i)} \bigl(j 2^{-l} \bigr) \Biggr)_{i \in\mathbb{N}} &\displaystyle \mathop{ \longrightarrow}_{n
\rightarrow\infty}^{(d)}& \Biggl( 2^{-l} \sum
_{j=1}^{4^l} \mu_{\xi(i)}
\bigl(C_{\sigma} j 2^{-l} \bigr) \Biggr)_{i \in\mathbb{N}},
\end{eqnarray*}
where $C_{\sigma} = 1+ 1/\sigma^2$. Using the upper bound (\ref
{EMajIntf}) and the first part of Lemma~\ref{TKEBr}, we get that
\begin{eqnarray*}
&&\lim_{l \rightarrow\infty} \sup_{n \in\mathbb{N}} \mathbb{E} \Biggl[
\Biggl\llvert\int_0^{\infty} \mu_{n,\xi_n (i)}
(t) \,dt - 2^{-l} \sum_{j=1}^{4^l}
\mu_{n,\xi_n (i)} \bigl(j2^{-l} \bigr)\Biggr\rrvert\Biggr] = 0,
\end{eqnarray*}
and these expectations do not depend on $i$. Proposition 3.1 of \cite
{BerMi} shows that $\delta(0,\xi(i))$ has the same law as $d (0,\xi
(i))$ and, therefore, has finite mean. Thus,
\begin{eqnarray*}
&&\Biggl\llvert\int_0^{\infty} \mu_{\xi(i)}
(C_{\sigma} t ) \,dt - 2^{-l} \sum_{j=1}^{4^l}
\mu_{\xi(i)} \bigl(C_{\sigma} j2^{-l} \bigr)\Biggr\rrvert
 \leq\underbrace{2^{-l} + \mathbb{E} \biggl[\int_{2^l}^{\infty}
\mu_{\xi(i)} (C_{\sigma} t ) \,dt \biggr]}_{\mathop{\longrightarrow}\limits_{l \rightarrow\infty} 0},
\end{eqnarray*}
and the left-hand side does not depend on $i$. Since
\begin{eqnarray*}
&&\int_0^{\infty} \mu_{\xi(i)}
(C_{\sigma} t ) \,dt = C_{\sigma}^{-1} \int
_0^{\infty} \mu_{\xi(i)} (t) \,dt =
C_{\sigma
}^{-1} \delta\bigl(0,\xi(i)\bigr),
\end{eqnarray*}
we conclude that
\begin{eqnarray*}
&&\bigl(C_{\sigma} \delta'_n \bigl(0,
\xi_n(i)\bigr) \bigr)_{i \in\mathbb
{N}} \mathop{ \longrightarrow}_{n \rightarrow\infty}^{(d)}\,
\bigl(\delta\bigl(0, \xi(i)\bigr) \bigr)_{i \in
\mathbb{N}},
\end{eqnarray*}
jointly with $(a_n/n) \cdot\mathcal{T}_n \mathop{ \longrightarrow
}\limits ^{(d)} \mathcal{T}$. Using in addition
the convergence of the $\tau_n (i,j)$ shown in Proposition~\ref
{TFstJointCv}, we see that the preceding convergences also hold jointly with
\begin{eqnarray*}
&&\bigl( C_{\sigma} \delta'_n \bigl(
\xi_n (i), \xi_n(j)\bigr) \bigr)_{i,j
\in\mathbb{N}} \mathop{ \longrightarrow}_{n \rightarrow\infty}^{(d)}\, \bigl(\delta\bigl(\xi(i), \xi(j)
\bigr) \bigr)_{i,j \in\mathbb{N}},
\end{eqnarray*}
and this gives the convergence (\ref{EJointCvCut}).

\setcounter{teo}{0}
\begin{appendix}\label{app}
\section*{Appendix: Adaptation of Doney's result}

We rephrase Lemma~\ref{TDoney} using the notation of \cite{Don}.

%
\begin{lem}
Let $(X_i)_{i \in\mathbb{N}}$ be a sequence of i.i.d.
variables in $\mathbb{N}\cup
\{0\}$, whose law belongs to the domain of attraction of a stable law
of index $\halpha\in(0,1)$, and $S_n = X_1 + \cdots+ X_n$. We also
let $A \in R_{\halpha}$ be a positive increasing function such that
%
%
\begin{equation}
\label{HDoney1} \mathbb{P} (X > r ) \sim\frac{1}{A (r)},
\end{equation}
and $a$ the inverse function of $A$. Besides, we suppose that the
additional hypothesis
%
%
\begin{equation}
\label{HDoney2} \sup_{r \geq1} \biggl(\frac{r \mathbb{P}(X = r)}{\mathbb
{P}(X >
r)} \biggr) <
\infty
\end{equation}
holds. Then there exist constants $B, C$ such that for all $r \in
\mathbb{N}
$, for all $n$ such that $r / a_n \geq B$,
\[
\mathbb{P} (S_n = r ) \leq C \frac{n}{r A (r)}.
\]
\end{lem}

This result is an adaptation of a theorem shown by Doney in \cite
{Don}, which gives an equivalent for $\mathbb{P} (S_n = r
)$ as $n\rightarrow
\infty$, uniformly in $n$ such that $r / a_n \rightarrow\infty$,
using the slightly stronger hypothesis
\[
\mathbb{P} (X=r ) \sim\frac{1}{r A (r)} \qquad\mbox{as } r \rightarrow
\infty
\]
instead of (\ref{HDoney2}).

\begin{pf*}{Sketch of the proof}
The main idea is to split up $\mathbb{P} (S_n = r )$ into
four terms, depending
upon the values taken by $M_n = \max\{X_i\dvtx  i=1, \ldots, n \}$ and
$N_n = \llvert\{ m \leq n\dvtx  X_m > z \}\rrvert$. More
precisely, letting
$\eta$ and $\gamma$ be constants in $(0,1)$, $w = r / a_n$ and $z =
a_n w^{\gamma}$, we have
\[
\mathbb{P} (S_n = r ) = \sum_{i=0}^3
\mathbb{P} \bigl(\{ S_n = r\} \cap A_i \bigr),
\]
where $A_i = \{M_n \leq\eta r, N_n = i \}$ for $i = 0,1$, $A_2 = \{M_n
\leq\eta r, N_n \geq2 \}$ and $A_3 = \{M_n > \eta r \}$. For our
purposes, it is enough to show that there exist constants $c_i$ such that
\[
q_i:= \mathbb{P} \bigl(\{S_n = r\} \cap
A_i \bigr) \leq c_i \frac
{n}{r A(r)} \qquad\forall i
\in\{0,1,2,3\}.
\]

The constants $\gamma$ and $\eta$ are fixed, with conditions that
will be given later (see the detailed version of the proof for explicit
conditions). In the whole proof, we suppose that $w \geq B$, for $B$
large enough (possibly depending on the values of $\eta$ and $\gamma
$). Note that hypotheses (\ref{HDoney1}) and (\ref{HDoney2}) imply
the existence of a constant $c$ such that
%
%
\begin{equation}
\label{EMajPr&Fr} p_r = \mathbb{P} (X=r ) \leq\frac{c}{r A(r)}
\quad\mbox{and}\quad\overline{F} (r) = \mathbb{P} (X>r ) \leq\frac
{c}{A (r)}.
\end{equation}

The first calculations of \cite{Don} show that we have the following
inequalities:
\begin{eqnarray*}
q_3 &\leq& n \sup_{l > \eta r} p_l,
\\
q_2 &\leq&\frac{1}{2} n^2 \overline{F} (z) \sup
_{l>z} p_l,
\\
q_1 &\leq& n \mathbb{P} \bigl(M_{n-1} \leq z,
S_{n-1} > (1-\eta) r \bigr) \sup_{l>z}
p_l.
\end{eqnarray*}
We now use (\ref{EMajPr&Fr}), and apply Lemma~\ref{TLemmeFVR} for the
regularly varying function $A$. The first inequality thus yields the
existence of a constant $c_3$ which only depends on the value of $\eta
$. Similarly, the second inequality gives the existence of $c_2$,
provided~$\gamma$ is large enough (independently of $B$) and $B \geq1$.

To get the existence of $c_1$, we first apply Lemma 2 of \cite{Don},
which gives an upper bound for the quantity $\mathbb{P} (M_{n-1}
\leq z, S_{n-1} > (1-\eta) r )$ provided $z$ is large enough and
$(1-\eta)r
\geq z$. Since $a_1 w^{\gamma} \leq z \leq r / w^{1-\gamma}$, these
conditions can be achieved by taking $B$ large enough. The lemma gives
\[
q_1 \leq c \frac{n}{z A(z)} \cdot\biggl(\frac{c' z}{(1-\eta)
r}
\biggr)^{(1-\eta)r/z}, 
\]
where $c'$ is a constant. Now, applying Lemma~\ref{TLemmeFVR}, we get
the existence of a constant $c'_1$ such that
\[
q_1 \leq c'_1 \frac{n}{r A(r)} \cdot
w^{\kappa},
\]
where $\kappa$ depends on the values of $\eta$, $\gamma$ and $B$.
For a given choice of $\eta$ and $\gamma$, and for $B$ large enough,
$\kappa$ is negative, hence the existence of $c_1$.

For $q_0$, getting the upper bound goes by first showing that we can
work under the hypotheses $r \leq nz$ and $r \leq n a_n/2$ (instead of
the hypotheses $n \rightarrow\infty$ and \mbox{$r / n a_n \rightarrow0$} of
\cite{Don}). Indeed, if $r > nz$, then $q_0 = 0$, and if $r > n
a_n/2$, another application of Lemma 2 of \cite{Don} and of Lemma~\ref
{TLemmeFVR} yields the result. The rest of the proof relies on
replacing the $X_i$ by truncated variables $\widehat{X}_i$, and using an
exponentially biased probability law. This last part is long and
technical, but it is rather easy to check that each step still holds
with our hypotheses, for $B$ large enough and with an appropriate
choice of $\eta$ (independently of $B$).
\end{pf*}
\end{appendix}

\section*{Acknowledgement}
I would like to thank Gr\'{e}gory
Miermont for many insightful comments and very thorough proofreadings.


%

\printaddresses
\end{document}